\documentclass[11pt,a4paper]{article}
\usepackage[utf8]{inputenc}
\usepackage{amsmath}
\usepackage{amsfonts}
\usepackage{amssymb}
\usepackage{amsthm}
\usepackage{hyperref}
\usepackage{enumitem}
\usepackage{xcolor}
\usepackage[left=3cm,right=3cm,top=3cm,bottom=3cm]{geometry}

\usepackage{tikz}
\usetikzlibrary{calc}

\tikzstyle{vertex}=[circle,draw, top color=gray!5, 
	    bottom color=gray!30, minimum size=10pt, scale=1, inner sep=0.5pt]

\tikzstyle{vertexsmall}=[circle,draw, top color=black, bottom color=black,
        minimum size=10pt, scale=1, inner sep=0.5pt]

\tikzstyle{vertexverysmall}=[circle,draw, top color=black, bottom color=black,
        minimum size=5pt, scale=1, inner sep=0.5pt]

\tikzstyle{edge}=[thick]

\newtheorem{theorem}{Theorem}
\newtheorem{corollary}[theorem]{Corollary}

\newtheorem{lemma}[theorem]{Lemma}

\newtheorem{claim}[theorem]{Claim}
\newtheorem{definition}[theorem]{Definition}
\newtheorem{problem}[theorem]{Problem}

\DeclareMathOperator{\bigO}{\mathcal{O}}
\DeclareMathOperator{\bigOmega}{\Omega}
\DeclareMathOperator{\dist}{\mathrm{dist}}
\newcommand{\wreach}{\mathrm{WReach}}
\newcommand{\wcol}{\mathrm{wcol}}
\newcommand{\diam}{\mathrm{diam}}

\providecommand{\noopsort}[1]{}

\title{Neighborhood complexity of planar graphs}

\author{Gwena\"el Joret\thanks{Supported by a CDR grant and a PDR grant from the Belgian National Fund for Scientific Research (FNRS).}\\
\small D\'epartement d'Informatique\\[-0.8ex]
\small Universit\'e libre de Bruxelles\\[-0.8ex] 
\small Brussels, Belgium\\
\small\tt gwenael.joret@ulb.be\\
\and
Clément Rambaud \\
\small D\'epartement d'Informatique\\[-0.8ex]
\small ENS Paris, PSL University\\[-0.8ex] 
\small Paris, France\\
\small Universit\'e Côte d'Azur\\[-0.8ex]
\small CNRS, Inria, I3S\\[-0.8ex] 
\small Sophia Antipolis, France\\
\small\tt clement.rambaud@inria.fr
}

\begin{document}

\maketitle

\begin{abstract}

Reidl, Sánchez~Villaamil, and Stravopoulos (2019) characterized graph classes of bounded expansion as follows: A class $\mathcal{C}$ closed under subgraphs has bounded expansion if and only if there exists a function $f:\mathbb{N} \to \mathbb{N}$ such that for every graph $G \in \mathcal{C}$, every nonempty subset $A$ of vertices in $G$ and every nonnegative integer $r$, the number of distinct intersections between $A$ and a ball of radius $r$ in $G$ is at most $f(r) |A|$. 
When $\mathcal{C}$ has bounded expansion, the function $f(r)$ coming from existing proofs is typically exponential. 
In the special case of planar graphs, it was conjectured by Soko{\l}owski (2021) that $f(r)$ could be taken to be a polynomial.  

In this paper, we prove this conjecture: For every nonempty subset $A$ of vertices in a planar graph $G$ and every nonnegative integer $r$, the number of distinct intersections between $A$ and a ball of radius $r$ in $G$ is $\bigO(r^4 |A|)$. 
We also show that a polynomial bound holds more generally for every proper minor-closed class of graphs. 
\end{abstract} 

\section{Introduction}

Nešetřil and Ossona~de~Mendez~\cite{nesetril_grad_2008} introduced the notion of graph classes with bounded expansion as a way to capture graphs that are sparse in a robust way: Very informally, the graphs should not only have linearly many edges but also all their minors with `bounded depth' should have linearly many edges (where the bound depends on the depth). 
Nowadays, there are several characterizations of graph classes with bounded expansion. 
These classes have been characterized using generalized coloring numbers~\cite{zhu_coloring_2009}, low treedepth colorings~\cite{nesetril_grad_2008},  $p$-centered colorings~\cite{nesetril_grad_2008}, and quasi-wideness~\cite{D10}, to name a few.  
See the textbook by Nešetřil and Ossona~de~Mendez~\cite{NOdM-book} and their more recent survey~\cite{SurveyND} for an overview of this rich area. 

In this paper, we are interested in the following characterization of graph classes with bounded expansion, due to
Reidl, Sánchez~Villaamil and Stravopoulos~\cite{RSS19}:
A class $\mathcal{C}$ of graphs closed under taking subgraphs
has \emph{bounded expansion} if and only if there exists a function 
$f:\mathbb{N} \to \mathbb{N}$ such that for every graph $G \in \mathcal{C}$,
for every nonempty subset $A \subseteq V(G)$ of vertices in $G$ and 
every nonnegative integer $r$,
\[
|\{N^r[v] \cap A \mid v \in V(G)\}| \leq f(r) |A|.
\]
Here, $N^r[v]$ denotes the {\em ball of radius $r$ around $v$}, namely, the set of vertices at distance at most $r$ from $v$ in $G$.  
When the function $f$ exists, we say that the class $\mathcal{C}$ has {\em bounded neighborhood complexity}, and the minimum such function $f$ is called the \emph{neighborhood complexity} of $\mathcal{C}$. 
(We remark that in some papers a different terminology is used: `bounded neighborhood complexity' is called `linear neighborhood complexity' instead, emphasizing the linear dependency on $|A|$. 
In this paper, `linear neighborhood complexity' has thus a different meaning, namely that $f$ can be taken in $\bigO(r)$.)

Neighborhood complexity has algorithmic applications. In particular, it has been used in~\cite{eickmeyer_neighborhood_2016} to build a linear kernel for {\sc Distance-$r$-Dominating-Set} in graph classes with bounded expansion.

However, the general bound given by Reidl, Sánchez~Villaamil and Stravopoulos
in~\cite{RSS19} is exponential, and no effective bound
for neighborhood complexity of some particular classes, such as planar graphs,
was known before this work. 
Using the fact that the family of all balls in a planar graph has VC-dimension at most $4$~\cite{chepoi2007covering} in combination with the Sauer-Shelah Lemma (Lemma~\ref{lemma:sauer_shelah}), it is not difficult to show that  $|\{N^r[v] \cap A \mid v \in V(G)\}| = \bigO(|A|^4)$ holds for every planar graph $G$, every nonempty subset $A$ of vertices, and every nonnegative integer $r$; see  Corollary~\ref{corollary:bound_polynomial_in_A}. 
For planar graphs, the best known upper bound is due to 
Sokołowski~\cite{sokolowski_bounds_2021}, who showed that for every planar graph $G$, every nonempty subset $A$ of vertices, and every nonnegative integer $r$, 
\[
|\{N^r[v] \cap A \mid v \in V(G)\}| = \bigO(r^7|A|^3).
\]

Thus, a bound polynomial in $r$ can be achieved at the price of a cubic dependence on $|A|$, instead of a linear one. 
Sokołowski~\cite{sokolowski_bounds_2021} conjectured that the neighborhood complexity of planar graphs is polynomial, that is, there exists
a constant $c$ such that $|\{N^r[v] \cap A \mid v \in V(G)\}| = \bigO(r^c|A|)$.

\subsection*{Our contributions}

Our main contribution is a proof of the above conjecture, that planar graphs have polynomial neighborhood complexity.  

\begin{theorem}
    \label{thm:neighborhood_complexity_planar}
    The neighborhood complexity of planar graphs is in $\bigO(r^4)$. 
\end{theorem}

Actually this upper bound holds for the slightly more general notion of `profile complexity':

\begin{theorem}
    \label{thm:profile_complexity_planar}
    The profile complexity of planar graphs is in $\bigO(r^4)$. 
\end{theorem}

{\em Profile complexity} is a refined version of neighborhood complexity, which we now introduce.
The {\em profile} of a vertex $u$ on a set $A$ of vertices at distance $r$
is the $|A|$-tuple of all the distances $\dist(u,a)$ for $a \in A$, where values larger than $r$ are replaced by $+\infty$.
Since the profile of $u$ on $A$ at distance $r$ determines the intersection $N^r[u] \cap A$, the number of distinct profiles at distance $r$ on $A$ is an upper bound on $|\{N^r[u] \cap A \mid u \in V(G)\}|$. 
In fact, most upper bounds on $|\{N^r[u] \cap A \mid u \in V(G)\}|$ in the literature are actually obtained by bounding the number of distinct profiles at distance $r$ on $A$.
This motivates the introduction of the notion of \emph{profile complexity} of a class $\mathcal{C}$ of graphs, defined as the minimum function $f$ (if it exists) such that for every $G \in \mathcal{C}$, for every nonempty set $A \subseteq V(G)$ and for every nonnegative integer $r$,
the number of distinct profiles on $A$ at distance $r$ is at most $f(r)|A|$.
Furthermore, under mild conditions on $\mathcal{C}$, a class $\mathcal{C}$ of graphs with bounded neighborhood complexity also has bounded profile complexity (c.f.\ Lemma~\ref{lemma:from_balls_to_profiles}).

As it turns out, the upper bound in Theorem~\ref{thm:profile_complexity_planar} is almost tight. 
In Section~\ref{section:lower_bounds}, we present an $\Omega(r^3)$ lower bound.

A second contribution of this paper is a proof that every proper minor-closed class of graphs has polynomial profile complexity, and thus polynomial neighborhood complexity. 

\begin{theorem}
    \label{thm:profile_complexity_excluded_minor} 
    For every integer $t\geq 1$ the profile complexity of $K_t$-minor-free graphs is in $\bigO_t(r^{t^2-1})$. 
\end{theorem} 

(The subscript in the notation $\bigO_t(\cdot)$ indicates that the hidden constant factor depends on $t$, and likewise for the other asymptotic notations in the paper.) 

We also consider graphs of bounded treewidth and graphs on surfaces, and give upper bounds on their profile complexities that are nearly tight, see Sections~\ref{section:bounded_treewidth},~\ref{section:surfaces}, and~\ref{section:lower_bounds}.  
Table~\ref{tab:summary_of_the_bounds} gives a summary of the main upper and lower bounds proved in this paper. 

\begin{table}[!ht]
    \centering
    \renewcommand{\arraystretch}{2.5}
    \begin{tabular}{| c  c  c  c |}
        \hline
        \bf Graph class & \bf Upper bound & \bf Lower bound & \bf Reference \\
        \hline 
        \sc Planar & $\bigO(r^4)$ & $\bigOmega(r^3)$ & 
        Thm.~\ref{theorem:NC_planar_degree_4} and~Thm.~\ref{theorem:construction_bounded_treewidth} \\
        \hline
        \sc Euler Genus $g$ & $\bigO(r^5)$ & $\bigOmega(r^3)$  &  Cor.~\ref{cor:neighborhood_complexity_surfaces} and~Thm.~\ref{theorem:construction_bounded_treewidth} \\
        \hline
        \sc Treewidth $t$ & 
        $\bigO_t(r^{2t})$  & $\bigOmega_t(r^{t+1})$ &  Cor.~\ref{corollary:NC_bounded_treewidth} and Thm.~\ref{theorem:construction_bounded_treewidth}\\
        \hline
        \sc $K_t$-minor-free & $\bigO_t(r^{t^2-1})$  & $\bigOmega_t(r^{t-1})$  &  Thm.~\ref{thm:NC_Kt_Minor_free} and Thm.~\ref{theorem:construction_bounded_treewidth}\\
        \hline        
    \end{tabular}
    \caption{Summary of our upper and lower bounds for the profile complexities of various graph classes.}
    \label{tab:summary_of_the_bounds}
\end{table}

While planar graphs, and more generally proper minor-closed classes of graphs, have polynomial profile complexity (and so polynomial neighborhood complexity), we show in Section~\ref{section:lower_bounds} that $1$-planar graphs, which
are graphs that can be drawn in the plane in such a way that every edge crosses at most one other edge,
have super-polynomial profile and neighborhood complexities. 
$1$-planar graphs are relatively close to planar graphs---in particular, they have $\bigO(\sqrt{n})$-size separators, where $n$ denotes the number of vertices---and are among the ``simplest'' non-minor closed graph classes within the realm of classes with bounded expansion. 
Thus, the latter result suggests that, besides proper minor closed graph classes, there are probably not many other bounded expansion classes of interest that have polynomial neighborhood complexity.

\subsection*{Related works}

Very recently, the relation between neighborhood complexity and twin-width has been studied in~\cite{bonnet_twin-width_2021} and independently in~\cite{przybyszewski_vc-density_2022}. 
In particular, it is shown in these two papers that graphs  with bounded twin-width have bounded neighborhood complexity. 
(We refer the reader to~\cite{bonnet_twin-width_2021, przybyszewski_vc-density_2022} for the definition of twin-width.) 
These bounds have been recently improved in~\cite{bonnet2023_NC_bounded_tww}. 

We remark that every proper minor-closed class of graphs has both bounded expansion (as is well known) and bounded twin-width (as proved in~\cite{BKTW21}). 
Thus, bounded neighborhood complexity for such a class also follows from the above mentioned results on twin-width. 
It is worth pointing out though that bounded expansion and bounded twin-width are two incomparable properties for graph classes, as neither implies the other. 
In particular, there are graph classes that do not have bounded expansion but have bounded twin-width (e.g.\ cographs). 
Also, let us emphasize that the characterization of bounded expansion classes in terms of bounded neighborhood complexity mentioned at the beginning of the introduction only applies to classes closed under subgraphs, which is not the case of bounded twin-width graphs. 

We also note that neighborhood complexity at distance $1$ and $2$ has also been used recently in~\cite{bonnet_reduced_2022} for problems related to twin-width.

We close this introduction by mentioning another concept that is closely related to neighborhood complexity.
The \emph{metric dimension} of a connected graph $G$ is the minimum size $k$ of a set $R$ of vertices in $G$ such that for every vertex $v$, the $k$-tuple
$(\dist_G(v,r))_{r \in R}$ uniquely determines $v$. 
Beaudou, Foucaud, Dankelmann, Henning, Mary, and Parreau~\cite{beaudou_bounding_2018} gave upper bounds on the number of vertices of graphs with given diameter and metric dimension. 
In particular, they showed that for every connected $n$-vertex graph $G$ with diameter $d$ and metric dimension $k$, 
if $G$ has treewidth at most $t$ then 
\[
n\in \bigO(d^{3t+3}k)
\] 
and if $G$ is $K_t$-minor free then 
\[
n \leq (dk+1)^{t-1}+1.
\] 
As it turns out, upper bounds on profile complexities can be translated to upper bounds in the latter setting, and as a result we obtain improved bounds on the number of vertices in the latter setting:  
if $G$ has treewidth at most $t$ then 
\[
n \in \bigO_t(d^{2t}k) 
\]
and if $G$ is $K_t$-minor free then 
\[
n \in \bigO_t(d^{t^2-1}k),
\]
see Theorem~\ref{thm:metric_dimension} in Section~\ref{sec:metric_dimension} (The latter bound is of course an improvement only when $k \gg d$.) 
As far as we are aware, this connection between neighborhood complexity and metric dimension has not been made before in the literature. 

\subsection*{Paper organization}

The paper is organized as follows. 
First, in Section~\ref{sec:prelim} we introduce all necessary definitions, as well as some preliminary lemmas. 
Then in Section~\ref{section:proper_minor_closed} we prove the aforementioned bound of $\bigO_t(r^{t^2-1})$ on the profile complexity of $K_t$-minor-free graphs (Theorem~\ref{thm:profile_complexity_excluded_minor}). 
For the special case of graphs of treewidth $t$ (which are $K_{t+2}$-minor free), we give an improved bound of $\bigO_t(r^{2t})$  in Section~\ref{section:bounded_treewidth}, see Corollary~\ref{corollary:NC_bounded_treewidth}. 
In Section~\ref{sec:product_structure}, we consider graphs admitting a so-called `product structure', which includes planar graphs, and show the existence of certain types of `guarding sets' in these graphs (to be used in Section~\ref{section:planar}), see Section~\ref{sec:product_structure} for the definitions. 
The next section, Section~\ref{section:planar}, contains our main result, namely a bound of $\bigO(r^{4})$ on the profile complexity of planar graphs (Theorem~\ref{thm:profile_complexity_planar}).  
Since planar graphs are $K_5$-minor free, a bound of $\bigO(r^{24})$ already follows Theorem~\ref{thm:profile_complexity_excluded_minor}. 
As a warm-up to the proof of the $\bigO(r^{4})$ bound, we first give two easier proofs in Section~\ref{section:planar}, giving bounds of respectively $\bigO(r^{11})$ and $\bigO(r^{6})$. 
Next, in Section~\ref{section:surfaces}, we lift the result for planar graphs to graphs on surfaces and obtain a $\bigO_g(r^{5})$ bound on the profile complexity of graphs of Euler genus $g$ (Corollary~\ref{cor:neighborhood_complexity_surfaces}). 
In Section~\ref{section:lower_bounds}, we prove the lower bounds mentioned in Table~\ref{tab:summary_of_the_bounds}. 
In Section~\ref{sec:metric_dimension}, we study the connection between profile complexity and metric dimension, and obtain a number of new bounds on metric dimension using  our results. 
Finally, we conclude the paper by emphasizing two open problems in  Section~\ref{sec:conclusion}.

\section{Preliminaries}
\label{sec:prelim}

All graphs considered in this paper are finite, simple, and undirected. 
We will use the following notations.
Let $G$ be a graph and let $u$ be a vertex of $G$.

\begin{itemize}
    \item $[n] = \{1, \dots ,n\}$ for every integer $n \geq 1$.
    \item $V(G)$ is the set of vertices of $G$, $E(G)$ its set of edges.
    \item For any set $X$ and integer $d \geq 0$, $\binom{X}{d} = \{Y \subseteq X \mid
    |Y|=d \}$ and $\binom{X}{\leq d} = \{Y \subseteq X \mid |Y| \leq d\}$.
    \item For any set $X$, $\mathcal{P}(X)$ is the set of all subsets of $X$.    
    \item $N_G(u) = \{v \in V(G) \mid uv \in E(G)\}$, and more generally
        for every $r \geq 0$ we define
        $N^r_G(u) = \{v \in V(G) \mid \dist_G(u,v)=r\}$.
        Similarly $N^r_G[u] = \{v \in V(G) \mid \dist_G(u,v) \leq r\}$. 
        Here, $\dist_G(u,v)$ denotes the distance between $u$ and $v$ in $G$. 
    \item If $P$ is a path, we denote by $\ell(P)$ its length, 
        which is defined as its number of edges.
\end{itemize}

A \emph{tree decomposition} of a graph $G$ is a pair $(T,(X_z)_{z \in V(T)})$,
where $T$ is a tree and $X_z \subseteq V(G)$ for every $z \in V(T)$, with the 
following properties:
\begin{enumerate}[label=(\roman*)]
    \item for every vertex $u \in V(G)$, the subgraph $T_u$ of $T$
        induced by $\{z \in V(T) \mid u \in X_z\}$ is a non empty subtree, and
    \item for every edge $uv \in E(G)$, $V(T_u) \cap V(T_v) \neq \emptyset$.
\end{enumerate}
We call the sets $X_z$ the \emph{bags} of the tree decomposition.
The \emph{width} of a tree decomposition is the maximum size of a bag minus one.
The \emph{treewidth} of a graph $G$ is defined to be the minimum width of a tree
decomposition of $G$.
For standard results on treewidth, see for example~\cite[Section 12.3]{diestel_book}.
We will in particular use the following facts:
\begin{itemize}
    \item if $C$ is a clique of $G$, then $C$ is included in a bag,
    \item for every edge $yz \in E(T)$, if $T_1$ (resp.\ $T_2$)
    denotes the connected component of $y$ (resp.\ $z$) in $T-yz$,
    then $X_y \cap X_z$ intersects every 
    $(\bigcup_{x \in V(T_1)} X_x, \bigcup_{x \in V(T_2)} X_x)$-path in $G$.
\end{itemize}

The tree decompositions considered in this paper will be arbitrarily (and implicitly)
rooted at a node $s \in V(T)$. This induces a root $t_u$ of the subtree $T_u$ for every
vertex $u \in V(G)$ (namely, the node of $T_u$ closest to $s$ in $T$). 
Moreover, we write $x <_T y$ for two distinct nodes 
$x,y \in V(T)$ if $x$ is in the (unique) $(y,s)$-path in $T$. 
This defines a partial order on $V(T)$.

Given a subset $A \subseteq V(G)$ of vertices, a vertex $u \in V(G)$ and an integer $r\geq 0$, 
the \emph{profile} of $u$ on $A$ at distance $r$ is the function 
$\pi_{r,G}[u \to A]$ that maps a vertex $a \in A$ to $\dist_G(u,a)$
if $\dist_G(u,a) \leq r$, to $+\infty$ otherwise.
Formally, $\pi_{r,G}[u \to A]: A \to \{0, \dots,r, +\infty\}, a \mapsto
\mathrm{Cap}_r(\dist_G(u,a))$, where $\mathrm{Cap}_r$ is the function that maps an integer $x$ to itself if $x \leq r$, to $+\infty$ otherwise.
For every $U \subseteq V(G)$, we denote by $\Pi_{r,G}[U \to A]$ 
the set of all profiles $\pi_{r,G}[u \to A]$ for vertices $u \in U$ except the `all $+\infty$' profile; that is, $\Pi_{r,G}[U \to A]=\{\pi_{r,G}[u \to A] \mid u \in U\} \setminus \{\lambda a.+\infty\}$. 
(We exclude the `all $+\infty$' profile because it is convenient in the proofs.) 
Clearly, $|\{ N^r[u] \cap A \mid u \in V(G)\}| \leq |\Pi_{r,G}[V(G) \to A]| + 1$, 
so it is enough to bound $|\Pi_{r,G}|V(G) \to A]|$ to get directly a bound on the
neighborhood complexity. The advantage of considering profiles is the following
``transitivity'' property: Given $U \subseteq V(G)$, if $S\subseteq V(G)$ meets 
every $(U,A)$-path of length at most $r$ in $G$, 
then the profile of a vertex $u \in U$ on $A$ is completely determined
by its profile on $S$.
This observation is also used in~\cite{RSS19} to prove graph classes with bounded expansion have bounded neighborhood complexity. 
This is a key idea, which justifies the introduction of our main tool,
the guarding sets, defined in Subsection~\ref{subsection:guarding_sets}.

\subsection{VC-dimension}

A \emph{set system} $\mathcal{F}$ over a universe $\Omega$
is a collection of subsets of $\Omega$.
The \emph{Vapnik–Chervonenkis dimension} (VC-dimension)~\cite{vapnik2015uniform}
of a set system $\mathcal{F}$ is the largest integer $d$ such that there 
exists $X \subseteq \Omega$ of size $d$ which is \emph{shattered} by $\mathcal{F}$:
for every $X' \subseteq X$, there exists $F \in \mathcal{F}$ such that $X \cap F = X'$.
Our first tool is a direct consequence of the well-known Sauer–Shelah Lemma \cite{SAUER1972145, shelah1972combinatorial}. 

\begin{lemma}[Corollary of Sauer–Shelah Lemma \cite{SAUER1972145, shelah1972combinatorial}]\label{lemma:sauer_shelah}
Let $\mathcal{F}$ be a set system over a universe $\Omega$ of VC-dimension
at most $d$ with $d \geq 2$. If $A \subseteq \Omega$ is non-empty then
\[
|\{A \cap F \mid F \in \mathcal{F}\}| \leq |A|^d.
\]
\end{lemma}

\begin{proof}
Let $n= |A|$.
The Sauer-Shelah Lemma states that $|\{A \cap F \mid F \in \mathcal{F}\}| \leq \sum_{i=0}^d
\binom{n}{i}$, so we only need to show the inequality
\[
\sum_{i=0}^d \binom{n}{i} \leq n^d.
\]
Consider $f:[n]^d \to \binom{[n]}{\leq d} \setminus \{\emptyset\}, 
(x_1, \dots, x_d) \mapsto \{x_i\}_{i=1,\dots,d}$.
Then $f$ is surjective but not injective (because $d \geq 2$) so
$|\binom{[n]}{\leq d}\setminus \{\emptyset\}| < n^d$ and thus
$\sum_{i=0}^d \binom{n}{i} \leq n^d$.
\end{proof}

We now prove that some particular set systems related to balls in planar graphs have bounded VC-dimension. 
Let us emphasize that, in the following lemma, $A$ is not an arbitrary vertex subset but is restricted to be a vertex subset of the outerface. This allows for a better bound on the VC-dimension (and this special case is sufficient for our  purposes). 

\begin{lemma}\label{lemma:VC_dim_outer_face_at_most_3}
Let $G$ be a plane graph and let $A$ be a subset of vertices of the outerface.
Then $\{N^r_G[v] \cap A \mid r \geq 0, v \in V(G)\}$ has VC-dimension at most $3$.
\end{lemma}

\begin{proof}
Suppose for contradiction that there exist $a,b,c,d \in A$
such that $\{a,b,c,d\}$ is shattered by some balls in $G$. 
Permuting $a,b,c,d$ if necessary, we may assume that there is a walk on the outerface such that $a,b,c,d$ appear for the first time in the walk in this order.
In particular, there exist $x_{ac},x_{bd} \in V(G)$ and
$r_{ac},r_{bd} \geq 0$ such that 
$N^{r_{ac}}_G[x_{ac}] \cap A = \{a,c\}$ and $N^{r_{bd}}_G[x_{bd}] \cap A = \{b,d\}$.

Let $P_{ac}$ (resp. $P_{bd}$) be the union of a shortest $(a,x_{ac})$-path
(resp. $(b,x_{bd})$-path) and a shortest $(x_{ac},c)$-path (resp. $(x_{bd},d)$-path).
By planarity, $P_{ac}$ and $P_{bd}$ have at least one common vertex $x'$. 
Permuting again $a,b,c,d$ if necessary (and possibly changing the orientation of the facial walk), we may assume that $x'$ is in
a shortest $(a,x_{ac})$-path $P'_{ac}$ and in a shortest $(b,x_{bd})$-path $P'_{bd}$. Moreover, we may also suppose that $\dist(x',a) \leq \dist(x',b)$.

As $a \not\in N^{r_{bd}}[x_{bd}]$ we have $\dist(x_{bd}, a) > r_{bd}$.
But $\dist(x_{bd},a) \leq \dist(x_{bd},x') + \dist(x',a) \leq \dist(x_{bd},x') +
\dist(x',b) = \dist(x_{bd}, b) \leq r_{bd}$. This contradiction completes the proof
of the lemma.
\end{proof}

The above lemma will be used later on in our proofs. 
We conclude this section with the following general theorem due to
Bousquet and Thomassé~\cite{bousquet_vc-dimension_2015}. 

\begin{theorem}[Bousquet and Thomassé~\cite{bousquet_vc-dimension_2015}]\label{theorem:VC_dim_Kt_minor_free}
Let $G$ be a $K_t$-minor-free graph, and let $\mathcal{H}$ be the set system on $V(G)$ consisting of all the balls in $G$ (of all possible radii). 
Then $\mathcal{H}$ has VC-dimension at most $t-1$.
\end{theorem}

Together with Lemma~\ref{lemma:sauer_shelah}, this gives directly
the following corollary.

\begin{corollary}\label{corollary:bound_polynomial_in_A}
Let $t$ be an integer with $t \geq 3$.
If $G$ is $K_t$-minor-free and $A \subseteq V(G)$ is nonempty, then 
\[
|\{N^r[v] \cap A \mid r \geq 0, v \in V(G)\}|  \leq |A|^{t-1}.
\]
\end{corollary}

\subsection{From balls to profiles}

Recall that the profile at distance $r$ of a vertex $u$ on a set $A$ of vertices in a graph $G$ is
the function $\pi_{r,G}[u \to A]:A \to \{0, \dots,r, + \infty\}$ that maps a vertex
$a \in A$ to $\dist_G(u,A)$ if this distance is at most $r$, to $+ \infty$ otherwise. 
Also, for $U\subseteq A$, $\Pi_{r,G}[U \to A]$ denotes
the set of all profiles $\pi_{r,G}[u \to A]$ for vertices $u \in U$ except the `all $+\infty$' profile. 
Thus, 
\[
|\{ N^r[u] \cap A \mid u \in V(G)\}| \leq |\Pi_{r,G}[V(G) \to A]| + 1.
\]
In this subsection, we show that the reverse inequality also holds in some approximate sense. 

\begin{lemma}\label{lemma:from_balls_to_profiles}
Let $\mathcal{C}$ be a class of graphs such that
\begin{enumerate}
\item if $G \in \mathcal{C}$, then every graph obtained from $G$ by adding a new vertex adjacent to exactly one vertex of $G$ is in $\mathcal{C}$ as well, and
\item there exists a function $f: \mathbb{N}^2 \to \mathbb{N}$ such that
    for every integer $r \geq 0$, for every graph $G \in \mathcal{C}$
    and every $A \subseteq V(G)$,
    \[
        |\{N^r_G[v] \cap A \mid v \in V(G)\}| \leq f(r,|A|).  
    \]
\end{enumerate}
Then, for every graph $G \in \mathcal{C}$ and every $A \subseteq V(G)$, we have
\[
|\Pi_{r,G}[V(G) \to A]| \leq f(r,(r+1)|A|).
\]

\end{lemma}

\begin{proof}
We build $G'$ from $G$ by gluing on every vertex $a \in A$ a path $a_0 a_1 \dots a_r$  of length $r$ with
$a_r=a$, and we let $A' = \{a_i \mid a \in A, 0 \leq i \leq r\}$ 
denote the union of the vertex set of these paths (including $A$). 
By hypothesis, $G'$ is in $\mathcal{C}$.

We claim that for every $u,v \in V(G)$, if $N^r_{G'}[u] \cap A'=N^r_{G'}[v] \cap A'$ then 
$\pi_{r,G}[u \to A] = \pi_{r,G}[v \to A]$.
Indeed, for every $a \in A$, $\pi_{r,G}[u \to A](a) = \min\{i \mid a_i \in N^r[u]\}=
\min\{i \mid a_i \in N^r[v]\} = \pi_{r,G}[v \to A](a)$ with the convention
$\min \emptyset = + \infty$.

Hence, the result follows from the fact that $|A'| = (r+1)|A|$.
\end{proof}

Combining Lemma~\ref{lemma:from_balls_to_profiles} with Corollary~\ref{corollary:bound_polynomial_in_A} gives the following corollary. 

\begin{corollary}\label{corollary:bound_polynomial_in_A_profiles} 
Let $t$ be an integer with $t\geq 3$ and let $G$ be a $K_t$-minor-free graph. Then, for every nonempty set $A \subseteq V(G)$ and
every integer $r \geq 0$
\[
|\Pi_{r,G}[V(G) \to A]| \leq (r+1)^{t-1}|A|^{t-1}
\]
In particular, if $G$ is planar, then $|\Pi_{r,G}[V(G) \to A]| \leq (r+1)^4|A|^4$.
\end{corollary}

We remark that a similar bound for metric dimension is proved in~\cite{beaudou_bounding_2018} using similar methods.

Using Lemma~\ref{lemma:VC_dim_outer_face_at_most_3}, we also deduce the following
bound.

\begin{corollary}\label{corollary:bound_profiles_outer_face}
If $G$ is a plane graph and $A \subseteq V(G)$ is non empty and is included in the
outerface of $G$, then for every integer $r \geq 0$
\[
|\Pi_{r,G}[V(G) \to A]| \leq (r+1)^3 |A|^3.
\]
\end{corollary}

\begin{proof}
Consider $G'$ as in the proof of Lemma~\ref{lemma:from_balls_to_profiles}. 
Observe that $G'$ is planar and can be drawn in the plane in such a way that $A'$ is included in its outerface.
Hence, by Lemma~\ref{lemma:VC_dim_outer_face_at_most_3} and Lemma~\ref{lemma:sauer_shelah}, 
\[
|\Pi_{r,G}[V(G) \to A]| \leq |\{ A' \cap N_{G'}^r[u] \mid u \in V(G)\}| \leq |A'|^3
\]
and the result follows as $|A'| = (r+1)|A|$.
\end{proof}

Corollaries~\ref{corollary:bound_polynomial_in_A_profiles} and~\ref{corollary:bound_profiles_outer_face} are our base blocks to derive polynomial bounds on the profile and neighborhood complexities. 
Most of the work will consist in reducing to the case where $A$ has size polynomial in $r$.

\subsection{Guarding sets}\label{subsection:guarding_sets}


\begin{definition}\label{definition:guarding_set}
Let $G$ be a graph, let $A \subseteq V(G)$,
and let $r,p\geq 0$ be integers.
We say that a family $\mathcal{S} \subseteq \mathcal{P}(V(G))$ is
an {\em $(r,p)$-guarding set for $A$ in $G$} if
\begin{enumerate}[label=(\roman*)]
\item $|S| \leq p$ for every $S \in \mathcal{S}$, and
\item for every vertex $v \in V(G)$, there exists $S \in \mathcal{S}$
    such that $S$ intersects every $(v,A)$-path of length at most $r$ in $G$ (if any).
    \label{item:def_guarding_cut_condition}
\end{enumerate}
\end{definition}

Typically we will be looking for an $(r,poly(r))$-guarding set of size $poly(r)|A|$.
This is a useful tool to bound the profile complexity thanks to the following property.

\begin{lemma}\label{lem:guarding_family_imply_small_NC}
Let $G$ be a graph, let $A \subseteq V(G)$, let $r,p\geq 0$ be integers, and let $\mathcal{S}$ be an $(r,p)$-guarding set for $A$. 
Suppose that, for some non-decreasing function $f$,  
\[
|\Pi_{r,G}[V(G) \to A']| \leq f(r,|A'|)
\]
for every $A' \subseteq V(G)$. 
Then
\[
|\Pi_{r,G}[V(G) \to A]| \leq f(r,p)|\mathcal{S}|.
\]
\end{lemma}

\begin{proof}
For every $S \in \mathcal{S}$, let $V_S$ be the set of vertices $v \in V(G)$
such that $S$ intersects every shortest $(v,a)$-path in $G$ for all $a \in A \cap
N^r[v]$.
The family $\{V_S\}_{S \in \mathcal{S}}$ covers $V(G)$
by~\ref{item:def_guarding_cut_condition}, 
hence $|\Pi_{r,G}[V(G) \to A]| \leq \sum_{S \in \mathcal{S}} |\Pi_{r,G}[V_S \to A]|$.

Now fix some $S \in \mathcal{S}$, and let us bound $|\Pi_{r,G}[V_S \to A]|$.
Let $v \in V_S$. For every $a \in A$ we have
\[
\begin{split}
\pi_{r,G}[v \to a] 
&= \mathrm{Cap}_r \circ \min_{s \in S}(\dist_G(v,s) + \dist_G(s,a)) \\
&= \mathrm{Cap}_r \circ \min_{s \in S}(\pi_{r,G}[v \to s] + \dist_G(s,a)).\\
\end{split}
\]

Therefore, $\pi_{r,G}[v \to A]$ is fully determined by $\pi_{r,G}[v \to S]
\in \Pi_{r,G}[V_S \to S]$. But by assumption, $|\Pi_{r,G}[V_S \to S]| \leq f(r,|S|)
\leq f(r,p)$. Hence $|\Pi_{r,G}[V(G) \to A]| \leq |\mathcal{S}| f(r,p)$ as claimed.
\end{proof}

We remark that Lemma~\ref{lem:guarding_family_imply_small_NC} still holds if
we relax the condition~\ref{item:def_guarding_cut_condition} in the definition of guarding sets by:
{\it
\begin{enumerate}
    \item[(ii')] for every $v \in V(G)$, there exists $S \in \mathcal{S}$ such that
    for every $a \in A$, either $\dist_G(v,a) > r$, or $S$ intersects at least one
    shortest $(v,a)$-path in $G$.
\end{enumerate}
}
However, all our constructions of guarding sets will satisfy~\ref{item:def_guarding_cut_condition} 
so we prefer to keep this stronger definition.

\subsection{Weak coloring numbers}

Given a graph $G$, a linear ordering $<$ on $V(G)$, an integer $r\geq 0$, and two vertices $u,v \in V(G)$,
we say that $v$ is \emph{$r$-weakly reachable} from $u$ if there exists a
$(u,v)$-path $P$ of length at most $r$ such that the minimum of $P$ for $<$ is $v$. 
We denote by $\wreach_r[G,<,u]$ the set of $r$-weakly reachable vertices from $u$
in $G$ for $<$. 
(Observe that this set contains $u$.)
The \emph{$r$-weak coloring number} of a graph $G$ is the minimum over all linear
orderings $<$ of $V(G)$ of the quantity $\max_{u \in V(G)} |\wreach_r[G,<,u]|$.

Weak coloring numbers give another characterization of graph classes with bounded expansion: Zhu~\cite{zhu_coloring_2009} proved that a class 
$\mathcal{C}$ of graphs has bounded expansion if and only if, for every $r\geq 0$,  the $r$-weak coloring number of every graph in $\mathcal{C}$ is bounded by a function of $r$ (which depends only on $\mathcal{C}$).

Weak coloring numbers will be a useful tool to design guarding sets when there is
a suitable tree decomposition, in particular in the proofs of
Theorems~\ref{theorem:guarding_bounded_treewidth}
and~\ref{theorem:guarding_product_general_form}.
We will use the following result on weak coloring numbers. 
(Let us recall that, given a tree decomposition $(T,(X_z)_{z \in V(T)})$ of a graph $G$, we consider $T$ to be rooted, and for $v\in V(G)$, we denote by $t_v$ the root of the subtree of $T$ induced by $v$.) 

\begin{theorem}[Grohe, Kreutzer, Rabinovich, Siebertz, and Stavropoulos~\cite{grohe_coloring_2018}]\label{theorem:wcol_bounded_tw} 
Let $t, r$ be nonnegative integers, let $G$ be a graph, and let $(T,(X_z)_{z \in V(T)})$ be a tree decomposition of $G$ of width at most $t$. 
Let $<$ be a linear ordering of $V(G)$ such that
for all $v,w \in V(G)$, if $t_v <_T t_w$, then $v<w$.
Then, for every vertex $v \in V(G)$, 
\[
|\wreach_r[G,<,v]| \leq \binom{r+t}{t} = \bigO_t(r^t).
\]
\end{theorem}

\section{Proper minor-closed classes of graphs}
\label{section:proper_minor_closed}

In this section we prove Theorem~\ref{thm:profile_complexity_excluded_minor}, which states that proper minor-closed graph classes  have polynomial profile complexity. 
Our proof is a modification of the proof in~\cite{RSS19} that graph classes with bounded expansion have bounded neighborhood complexity. 
The latter proof uses weak coloring numbers as their main tool to bound neighborhood complexity. 
While weak coloring numbers can be exponential for graph classes with bounded expansion, it is known that graphs excluding a fixed minor have polynomial weak coloring numbers:
\begin{theorem}[\cite{heuvel_generalised}]\label{thm:wcol_Kt_minor_free}
    For every positive integers $t$ and $r$ with $t \geq 4$, and for every $K_t$-minor-free graph $G$,
    \[
    \wcol_r(G) \leq \binom{r+t-2}{t-2} (t-3) (2r+1) \leq 2(t-3)(r+1)^{t-1}.
    \]
\end{theorem}
However, this fact alone is not enough to turn the proof from~\cite{RSS19} into a proof that graphs excluding a minor have polynomial neighborhood complexity. 
This is because in one of the last steps of the proof in~\cite{RSS19}, one needs to enumerate all subsets of a given set of vertices, which inherently results in a bound on neighborhood complexity that is exponential. 
In our proof, we avoid this costly step and use instead an argument based on VC-dimension to finish the proof. 
 
Theorem~\ref{thm:profile_complexity_excluded_minor}  follows from the following result. 

\begin{theorem}\label{thm:NC_Kt_Minor_free}
    For every positive integers $t$ and $r$ with $t \geq 4$, for every $K_t$-minor-free graph $G$,
    for every set $A$ of vertices of $G$,
    \[
    |\Pi_{r,G}[V(G) \to A]| \leq (2^t(t-3))^t \cdot (r+1)^{t^2-1}|A|.
    \]
\end{theorem}

\begin{proof}
     Using Theorem~\ref{thm:wcol_Kt_minor_free}, 
    let $<$ be an ordering of $V(G)$ witnessing the fact that 
    \[
    \wcol_{2r}(G) \leq 2(t-3)(2r+1)^{t-1} \leq 2^t(t-3) (r+1)^{t-1}.
    \]
    Let \[
    B = \bigcup_{a \in A} \wreach_r[G,<,a].
    \]
    Observe that $|B| \leq 2^t(t-3)(r+1)^{t-1} |A|$.
    For every vertex $b \in B$, let 
    \[
    S_b = \wreach_{2r}[G,<,b].
    \]
    For every vertex $u \in V(G)$ at distance at most $r$ from a vertex in $A$, let
    \[
    \phi(u) = \max_{<} \big(B \cap \wreach_r[G,<,u]\big).
    \]
    \begin{claim}
        For every vertex $u \in V(G)$ at distance at most $r$ from $A$, $S_{\phi(u)}$ intersects every $(u,A)$-path of length at most $r$ in $G$.
    \end{claim}
    \begin{proof}
        Let $Q$ be a path from $u$ to $\phi(u)$ witnessing the fact that $\phi(u) \in \wreach_r[G,<,u]$.
        Consider a path $P$ of length at most $r$ from $u$ to some vertex $a \in A$.
        Let $w = \min_< V(P)$, let $P_1$ be the subpath of $P$ from $u$ to $w$ and let $P_2$ the subpath of $P$ from $w$ to $a$.
        Then $P_1$ witnesses the fact that $w \in \wreach_r[G,<,u]$ and $P_2$ witnesses the fact that $w \in \wreach_r[G,<,a] \subseteq B$.
        Hence $w \in B \cap \wreach_r[G,<,u]$.
        In particular, $w \leq \phi(u)$ by definition of $\phi(u)$.
        Then the concatenation of $Q$ with $P_1$ witnesses the fact that $w \in \wreach_{2r}[G,<,\phi(u)] = S_{\phi(u)}$, and so
        $S_{\phi(u)}$ intersects $V(P)$ as claimed.
    \end{proof}

    Hence $\{S_{w} \mid w \in B\}$ is an $(r,2^t(t-3)(r+1)^{t-1})$-guarding set since $|S_w| \leq 2^t(t-3)(r+1)^{t-1}$ for every $w \in B$. 
    
    Now, we use the fact that $K_t$-minor-free graphs have bounded VC-dimension: By Corollary~\ref{corollary:bound_polynomial_in_A_profiles}, 
    \[|
    \Pi_{r,G}[V(G) \to S]| \leq (r+1)^{t-1}|S|^{t-1} \leq (2^t(t-3))^{t-1}(r+1)^{t^2-t}
    \] 
    for every set $S$ of at most $2^t(t-3)(r+1)^{t-1}$ vertices of $G$. 
    Using Lemma~\ref{lem:guarding_family_imply_small_NC}, we deduce that    
    \[
    \begin{split}
    |\Pi_{r,G}[V(G) \to A]| &\leq (2^t(t-3))^{t-1}(r+1)^{t^2-t} \cdot |\{S_{w} \mid w \in B\}|  \\
    &\leq (2^t(t-3))^{t-1}(r+1)^{t^2-t} \cdot |B| \\
    &\leq (2^t(t-3))^{t-1}(r+1)^{t^2-t} \cdot 2^t(t-3) (r+1)^{t-1} |A| \\
    &= (2^t(t-3))^t (r+1)^{t^2-1} |A|,
    \end{split}
    \]
    which concludes the proof of the theorem.
\end{proof}

\section{Graphs of bounded treewidth}\label{section:bounded_treewidth}

Since graphs of treewidth at most $t$ are $K_{t+2}$-minor-free, Theorem~\ref{thm:NC_Kt_Minor_free} implies that these graphs have profile complexity in $\bigO_t(r^{t^2+2t})$.
In this section, we show the more precise upper bound $\bigO_t(r^{2t})$.

\begin{theorem}\label{theorem:guarding_bounded_treewidth}
Let $t, r$ be nonnegative integers, let $G$ be a graph of treewidth at most $t$, and let $A \subseteq V(G)$.  
Then, there is an $(r,t)$-guarding set for $A$ in $G$ of size at most $(t+1)\binom{r+t}{t}|A|$.
\end{theorem}

\begin{proof}
Consider a tree decomposition $(T,(X_z)_{z \in V(T)})$ of width at most $t$
rooted at an arbitrarily chosen node $s \in V(T)$ and a linear order $<$ on $V(G)$ 
as in Theorem~\ref{theorem:wcol_bounded_tw}.
In addition, we suppose that for every edge $yz \in E(T)$, $|X_y \cap X_z| \leq
t$ (otherwise $X_y=X_z$ and we can contract $yz$). 
As before, we denote by $t_v$ the root of the subtree $T_v$ for every vertex $v \in V(G)$.

Let $A' = \bigcup_{a \in A} \wreach_r[G,<,a]$ and $B = \{t_v \mid v \in A'\}$.
Finally, we set $\mathcal{S} = \{X_z \mid z \in B, |X_z|\leq t \} \cup \bigcup_{z \in B, |X_z|=t+1} \binom{X_z}{t}$.
Clearly $|\mathcal{S}| \leq \binom{t+1}{t} |A| \cdot \max_{a \in A} \wreach_r[G,<,a]$ and so by Theorem~\ref{theorem:wcol_bounded_tw} $|\mathcal{S}| \leq (t+1)\binom{r+t}{t}|A|$.
Moreover, for every $S \in \mathcal{S}$, $|S| \leq t$, so it only remains
to show that for every vertex $v \in V(G)$, there exists $S \in \mathcal{S}$
that intersects every $(v,A)$-path of length at most $r$ in $G$. 
Let thus $v \in V(G)$, and assume that $v$ is at distance at most $r$ from $A$ (otherwise, there is nothing to show). 

We find the desired set $S$ as follows:
Consider the unique $(t_v,s)$-path $P$ in $T$, and let $z$ be the first node in $P$ that belongs to $B$. 
Such a node $z$ exists, which can be seen as follows: 
Consider a $(v,A)$-path $Q$ of length at most $r$ in $G$. Let $w$ be the minimum of $Q$. 
Then $w \in \wreach_r[G,<,a]$ for some $a \in A$, 
which implies $t_w \in B$. 
Furthermore, $\{y \in V(T) \mid V(Q) \cap X_y \neq \emptyset\}$ induces a connected subgraph of $T$, thus a subtree of $T$, and by our choice of the ordering $<$, $t_w$ is the root of this subtree. 

If $z=t_v$, then take for $S$ an arbitrary subset of $X_{t_v}$ of size $t$ that
contains $v$ if $|X_{t_v}| =t+1$, or $S = X_{t_v}$ if $|X_{t_v}| \leq t$, 
and the result is clear.
Otherwise, let $y$ be the predecessor of $z$ in $P$, and take
$S \subseteq X_z$ of size at most $t$ such that $X_y \cap X_z \subseteq S$ if $|X_z|=t+1$,
and take $S=X_z$ if $|X_z| \leq t$. 
In both cases, $S \in \mathcal{S}$.
Now consider a $(v,a)$-path $Q$ of length at most $r$ in $G$ for some $a \in A$.
Then $Q$ contains a minimal vertex $w$ for $<$, and so $t_w$ is in $P$.
Moreover, $w \in \wreach_r[G,<,a]$ so $w \in A'$ and $t_w \in B$. 
It follows that the node $z$ is on the $(t_v,t_w)$-path in $T$.  
Given our choice of $S$, the set $S$ intersects every $(v,w)$-path, and in particular intersects $Q$.
This concludes the proof that $\mathcal{S}$ is an $(r,t)$-guarding set for
$A$ in $G$ of size at most $(t+1)\binom{r+t}{t}|A|$.
\end{proof}

Combining Theorem~\ref{theorem:guarding_bounded_treewidth} with Lemma~\ref{lem:guarding_family_imply_small_NC} 
and the fact that $|\Pi_{r,G}[V(G) \to A']| \leq (r+2)^{|A'|}$ for every $A' \subseteq V(G)$, 
we deduce that graphs of bounded treewidth have polynomial neighborhood
complexity.

\begin{corollary}\label{corollary:NC_bounded_treewidth}
For every graph $G$ of treewidth at most $t$, every $A \subseteq V(G)$, and every nonnegative integer $r$, 
\[
|\Pi_{r,G}[V(G) \to A]| \leq (t+1)(r+2)^t \binom{r+t}{t} |A| = \bigO_t(r^{2t}|A|)
.
\]
\end{corollary}

\section{Graphs admitting a product structure}
\label{sec:product_structure}

The \emph{strong product} of two graphs $A$ and $B$ is the graph
$A \boxtimes B$ with vertex set $V(A) \times V(B)$  
where two distinct vertices $(a,b),(a',b')$ are adjacent if 
\[
(a=a' \text{ or } aa' \in E(A)) \text{ and } (b=b' \text{ or } bb' \in E(B)).
\] 
For example, the strong product of two paths is a grid with the diagonals in each square. 
Our interest in strong products of graphs comes from the fact that planar graphs have the following `product structure'. 

\begin{theorem}[Dujmović, Joret, Micek, Morin, Ueckerdt, and Wood~\cite{dujmovic_planar_2020}]\label{theorem:product_structure_planar}
    For every planar graph $G$, there is a graph $H$ of treewidth at most $3$ and a path $P$ such that $G$ is isomorphic to a subgraph of $H  \boxtimes P \boxtimes K_3$.
\end{theorem}

In this section, we lift the construction of guarding sets in graphs of bounded
treewidth to graphs having such a product structure.

Theorem~\ref{theorem:guarding_bounded_treewidth} can be extended to graphs with a
product structure as follows.

\begin{theorem}\label{theorem:guarding_product_general_form} 
Let $t, r$ be nonnegative integers.  
Let $H$ be a graph of treewidth at most $t$, let $c$ be a positive integer and let $P$ be a path. 
Let $G \subseteq H  \boxtimes P \boxtimes K_c$ and let $A \subseteq V(G)$. 
Then, $G$ has a $(r,4c(t+1)r)$-guarding set of size at most $2\binom{r+t}{t}|A|$.  
\end{theorem}

\begin{proof}
Let $(T,(X_z)_{z \in V(T)})$ be a tree decomposition of $H$ of width at most $t$
rooted at an arbitrary node $s\in V(T)$. 

Each vertex of $P$ defines a corresponding {\em row} of the product $H \boxtimes P$ in a natural way. 
Ordering the vertices of $P$ from one of its ends to the other, consider the corresponding rows, and let $L_0, L_1, \dots, L_p$ denote their intersections with $V(G)$, which we call a {\em layering} of $G$.  
Observe that, for every edge $uv\in E(G)$, the two vertices $u,v$ are either in the same layer or in consecutive layers. 
Observe also that every ball of radius $r$ in $G$ is included in at least one of the intervals 
\[
I_j = \bigcup_{k=0}^{4r-1}L_{2rj+k}
\]
for $j\geq 0$. (We use the convention that $L_i=\emptyset$ for $i>p$.)

Fix one of these intervals $I_j$, and let $A_j = A \cap I_j$.
For every $a \in A_j$, we denote by $u_a$ the (unique) vertex in $H$ such that 
$a \in V(H[\{u_a\}] \boxtimes P \boxtimes K_c )$.
Let $A'_j = \bigcup_{a \in A_j} \wreach_r[H, <, u_a] \subseteq V(H)$ where $<$ is as
in Theorem~\ref{theorem:wcol_bounded_tw}.
By Theorem~\ref{theorem:wcol_bounded_tw}, $|A'_j| \leq \binom{r+t}{t}|A_j|$.

We set $\mathcal{S}_j = \{V(H[X_{t_u}] \boxtimes P \boxtimes K_c ) \cap I_j\mid u \in A'_j\}$,
and finally $\mathcal{S} = \bigcup_{j \geq 0} \mathcal{S}_j$.

First, we have $|S| \leq (t+1)\cdot c \cdot 4r$ 
for every $S \in \mathcal{S}$,  by construction.
Second, $|\mathcal{S}| \leq \sum_{j\geq 0} |\mathcal{S}_j| \leq \sum_{j \geq 0} \binom{r+t}{t}|A_j|
\leq 2\binom{r+t}{t}|A|$ (since each $a \in A$ is included in $A_j$ for at most two distinct indices $j$).

It remains to show property~\ref{item:def_guarding_cut_condition}.
Let $v \in V(G)$ be a vertex such that $\dist_G(u,A) \leq r$.
There exists $j$ such that $N^r[v] \subseteq I_j$,
and so $A \cap N^r[v] \subseteq A_j$.
Let $z \in V(T)$ be such that $v \in X_z$.
Consider the (unique) $(z, s)$-path $P_{z, s}$ in $T$, and let
$y$ be the first node in $P_{z, s}$ that also belongs to 
$\{t_u \mid u \in A'_j\}$.
Note that the node $y$ is well defined because there is at least one $(v,A)$-path
$Q$ of length at most $r$ in $G$, and the path $Q$ induces a connected subgraph $\Tilde{Q}$ of $H$,
which has a minimum $w$ with respect to $<$. 
Moreover, $w \in A'_j$ (as follows from the definition of $A'_j$ and the existence of $\Tilde{Q}$), 
and $t_w \in V(P_{z, s})$ (because of our specific ordering $<$). 

We take $S = V(H[X_y] \boxtimes P \boxtimes K_c ) \cap I_j$, which is in  $\mathcal{S}_j$.
We claim that $S$ intersects every $(v,A)$-path of length at most $r$ in $G$. 
The argument is similar to the one showing the existence of $y$ above: 
For every $a \in A$ and every $(v,a)$-path $Q$ of length at most $r$,
we have $V(Q) \subseteq N^r[v] \subseteq I_j$.
Moreover, $Q$ induces in $H$ a connected subgraph $\Tilde{Q}$, which has a minimum $w$ with respect to $<$.  
Then $t_{w}$ belongs to the unique $(z,s)$-path in $T$,
and it follows that $X_y$ intersects $V(\Tilde{Q})$ by our choice of $y$, and so $S$ intersects $Q$.
This proves property~\ref{item:def_guarding_cut_condition}, and concludes the
proof of the theorem.
\end{proof}

\section{Planar graphs}\label{section:planar}

Since planar graphs are $K_5$-minor-free, Theorem~\ref{thm:NC_Kt_Minor_free} directly implies that planar graphs have profile complexity (and so neighborhood complexity) in $\bigO(r^{24})$. 
When restricting the proof of Theorem~\ref{thm:NC_Kt_Minor_free} to planar graphs, one may use a better bound of $\bigO(r^3)$ for their weak coloring numbers (as proved in~\cite{heuvel_generalised}), which results in a better upper bound of $\bigO(r^{16})$ for the profile complexity of planar graphs. 
In this section we improve on these bounds.

We give three proofs.
The first one is very short, it combines the product structure of planar graphs~\cite{dujmovic_planar_2020} mentioned in the previous section together with Theorem~\ref{theorem:guarding_product_general_form} and gives a $\bigO(r^{11})$ bound on the profile complexity. 
The second one starts by applying a lemma of Soko\l{}owski~\cite{sokolowski_bounds_2021} to reduce to the case where the set $A$ under consideration is the vertex set of the outerface, and also relies on the existence of `sparse covers' for planar graphs (see below). 
It results in a $\bigO(r^6)$ bound. 
The third proof is a variant of the second one, giving a $\bigO(r^4)$ bound. 
It uses two extra ingredients, a recent theorem of Li and Parter~\cite{li_parter_stoc_2019} (Theorem~\ref{theorem:base_bound_Li_Parter} below)  and some form of decomposition of planar graphs using `tripods' that is directly inspired from the proof of the product structure of planar graphs~\cite{dujmovic_planar_2020}. 
We note that this third proof resulted from discussions with Jacob Holme, Erik Jan van Leeuwen, and Marcin Pilipczuk, who upon reading an earlier version of this paper pointed out reference~\cite{li_parter_stoc_2019} to us, and the fact that it could be used to improve the $\bigO(r^6)$ bound to $\bigO(r^4)$. 
We are grateful to them for their helpful feedback. 

\subsection{First proof}

We start with the $\bigO(r^{11})$ bound on the profile complexity of planar graphs. 

\begin{theorem}\label{theorem:NC_planar_degree_16}
If $G$ is a planar graph, $r$ is a nonnegative integer, and $A$ is a nonempty set of vertices of $G$, 
then
\[
|\Pi_{r,G}[V(G) \to A]| \leq 10^8 (r+1)^{11}|A|.
\]
\end{theorem}

\begin{proof}
By Theorems~\ref{theorem:product_structure_planar} and~\ref{theorem:guarding_product_general_form}, $G$ has an $(r,48r)$-guarding
set $\mathcal{S} \subseteq \mathcal{P}(V(G))$ of size at most $2\binom{r+3}{3}|A|$.
It follows from Lemma~\ref{lem:guarding_family_imply_small_NC} and Corollary~\ref{corollary:bound_polynomial_in_A_profiles} that
\[
|\Pi_{r,G}[V(G) \to A]| \leq (r+1)^4(48r)^4|\mathcal{S}| \leq (48)^4 (r+1)^8 2\binom{r+3}{3}|A| \leq 2(48)^4 (r+1)^{11} |A|
\]
and so $|\Pi_{r,G}[V(G) \to A]| \leq 10^8 (r+1)^{11}|A|$.
\end{proof}

\subsection{Second proof}

We begin this second proof with a lemma  similar to Lemma~\ref{lem:guarding_family_imply_small_NC}.

\begin{lemma}\label{lemma:covers_and_cuts}
Let $G$ be a graph, let $A$ be a nonempty set of vertices of $G$, and let $r$ be a nonnegative integer. 
Suppose that $V_0,V_1, \dots, V_q$ is a collection of subsets of $V(G)$ with $\bigcup_{i=0}^q V_i = V(G)$
and $S_0, \dots, S_q$ is a collection of set of vertices
such that for every
$i\in [0, q]$ $S_i \subseteq V_i$, and for every vertex $v \in V(G)$, there exists $i \in [0,q]$ so that $v \in V_i$ and for every $a \in A \cap N^r_G[v]$, there is a shortest $(v,a)$-path $P$ in $G$ whose maximal prefix in $V_i$ intersects $S_i$
(informally, $P$ must intersect $S_i$ before leaving $V_i$).
Then
\[
|\Pi_{r,G}[V(G) \to A]| \leq
\sum_{i=0}^q |\Pi_{r,G[V_i]}[V_i \to S_i]|.
\]
\end{lemma}

\begin{proof}
Let $v \in V(G)$. Let $i \in [0,q]$ be as in the statement,
then for every $a \in A$
\[
\begin{split}
\pi_{r,G}[v \to a] 
&= \mathrm{Cap}_r \circ \min_{s \in S_i}(\dist_{G[V_i]}(v,s) + \dist_G(s,a)) \\
&= \mathrm{Cap}_r \circ \min_{s \in S_i}(\pi_{r,G[V_i]}[v \to s] + \dist_G(s,a)).\\
\end{split}
\]
It follows that $|\Pi_{r,G}[V(G) \to A]| \leq
\sum_{i= 0}^q |\Pi_{r,G[V_i]}[V_i \to S_i]|$.
\end{proof}

The following lemma due to Soko\l{}owski~\cite{sokolowski_bounds_2021} is a reduction for planar graphs to the case where the vertices in $A$ are exactly the vertices of the outerface, at the cost of a small increase in the size of $A$. 
It is a key ingredient of our second and third proofs. 

\begin{lemma}[Sokołowski~\cite{sokolowski_bounds_2021}]\label{lemma:reduction_outerface}
Let $G$ be a connected planar graph, let $A$ be a nonempty set of vertices of $G$, and let $r$ be a nonnegative integer. 
If every vertex of $G$ is at distance at most $r$ from $A$, then one can construct a plane graph $G'$ and a set $A' \subseteq V(G')$ such that
\begin{enumerate}[label=(\roman*)]
\item $|A'| \leq 2(2r+1)|A|$;
\item the outerface of $G'$ is bounded by a cycle;
\item $A'$ is the vertex set of the outerface of $G'$, and
\item $|\Pi_{r,G}[V(G) \to A]| \leq |\Pi_{r,G'}[V(G') \to A']|$.
\end{enumerate}
\end{lemma}

Since this lemma is not explicitly stated in~\cite{sokolowski_bounds_2021}, but rather used implicitly in a proof (see proof of Theorem 26 in~\cite{sokolowski_bounds_2021}), we sketch here its elegant proof.
\begin{proof}[Sketch of proof]
Since $G$ is connected and every vertex in $G$ is at distance at most $r$ from a vertex in $A$, we deduce that there is a tree $T$ in $G$ such that $A \subseteq V(T)$ and $|E(T)| \leq (2r+1)(|A|-1)$. Such a tree can be obtained by considering a tree $T$ with $A \subseteq V(T)$ of minimal size.

The graph $G'$ is now obtained from $G$ by ``cutting  the plane open'' along $T$:
the tree $T$ is replaced by a face $A'$ whose vertices
are the occurrences of the vertices in $T$ along an Euler tour of $T$ compatible with the embedding. See Figure~\ref{fig:cut_along}.
Since every vertex $u$ in $T$ corresponds to $d_{T}(u)$ vertices in $A'$, we have $|A'| = \sum_{u \in V(T)} d_T(u) = 2|E(T)| \leq 2(2r+1)|A|$.
Moreover, one can check that for every $x,y \in V(G) \setminus A$, if $x$ and $y$ have the same profile at distance $r$ on $A'$ in $G'$, then $x$ and $y$ have the same profile
at distance $r$ on $A$ in $G$. We conclude that
\[
\begin{split}
|\Pi_{r,G}[V(G) \to A] 
&\leq |A| + |\Pi_{r,G}[V(G) \setminus A \to A]| \\
&\leq |A| + |\Pi_{r,G'}[V(G') \setminus A' \to A']| \\
& \leq |A'| + |\Pi_{r,G'}[V(G') \setminus A' \to A']| \\
&\leq |\Pi_{r,G'}[V(G') \to A']|. \\
\end{split}
\]
\end{proof}

\begin{figure}[ht]
    \centering
    \includegraphics{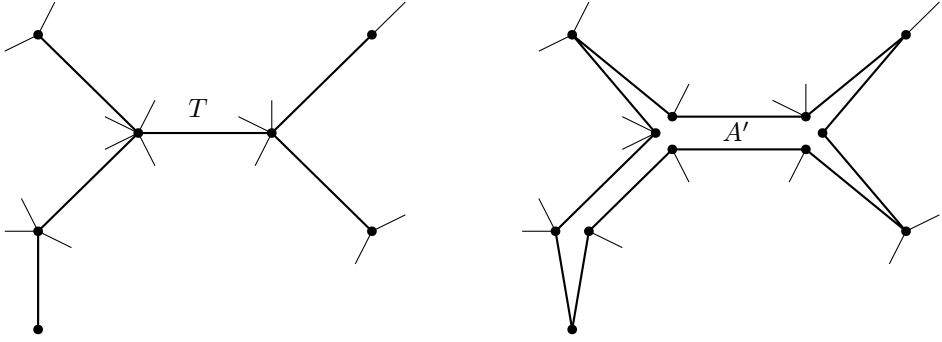}
    \caption{The construction of $G'$ and $A'$ in the proof of Lemma~\ref{lemma:reduction_outerface}}
    \label{fig:cut_along}
\end{figure}

We will also use the following objects called `sparse covers', 
which will allow us to focus on graphs with radius $\bigO(r)$ when bounding the profile complexity at distance $r$. 
Let $G$ be a graph. 
Given integers $r,d,k \geq 0$, a family $\mathcal{X} \subseteq \mathcal{P}(V(G))$ is an
{\em $(r,d,k)$-sparse cover of $G$} if
\begin{enumerate}[label=(\roman*)]
\item for every vertex $v \in V(G)$, there exists $X \in \mathcal{X}$ such that $N^r_G[v] \subseteq X$,
\item for every $X \in \mathcal{X}$, the graph $G[X]$ has radius at most $d$, and
\item for every vertex $v \in V(G)$, we have $|\{X \in \mathcal{X}\mid v \in X\}| \leq k$.
\end{enumerate}

The following two theorems provide sparse covers for planar graphs and graphs excluding a minor.

\begin{theorem}[Busch, LaFortune, and Tirthapura~\cite{busch_improved_2007}]\label{theorem:construct_X_planar}
For every planar graph $G$ and nonnegative integer $r$, $G$ has an $(r,24r,18)$-sparse cover.
\end{theorem}

\begin{theorem}[Abraham, Gavoille, Malkhi, and Wieder~\cite{abraham_strong-diameter_2007}]\label{theorem:construct_X_minor_closed}
For every $K_t$-minor-free graph $G$ and nonnegative integer $r$,
$G$ has an $(r,\bigO(t^2r),2^{\bigO(t)}t!)$-sparse cover.
\end{theorem}

We are now ready to prove a $\bigO(r^6)$ bound on the profile complexity of
planar graphs.

\begin{theorem}\label{theorem:NC_planar_degree_6}
Let $G$ be a planar graph, let $A$ be a nonempty set of vertices of $G$, and let $r$ be a nonnegative integer. Then, 
\[
|\Pi_{r,G}[V(G) \to A]| \leq 10^8 (r+1)^6 |A|.
\]
\end{theorem}

\begin{proof}
If some vertex is at distance more than $r$ from $A$, then we can remove it
without decreasing the size of $\Pi_{r,G}[V(G) \to A]$
(remember that for convenience, the profile which is  $+\infty$ everywhere does not belong to this set).
Thus, we may assume that every vertex in $G$ is at distance at most $r$ from
at least one vertex of $A$.
Moreover, each connected component of $G$ can be considered separately (with the restriction of $A$ to the component), so we may suppose that $G$ is connected. 

By Lemma~\ref{lemma:reduction_outerface}, there exists a plane graph $G'$
and $A' \subseteq V(G)$ such that
\begin{enumerate}[label=(\roman*)]
\item $|A'| \leq 2(2r+1)|A|$;
\item the outerface of $G'$ is bounded by a cycle $C'$;
\item $A'$ is the vertex set of the outerface of $G'$, and
\item $|\Pi_{r,G}[V(G) \to A]| \leq |\Pi_{r,G'}[V(G') \to A']|$.
\end{enumerate}

Next, apply Theorem~\ref{theorem:construct_X_planar} to $G'$, and let $\mathcal{X}$ be an $(r,24r,18)$-sparse cover of $G'$. 
Our goal is to bound $|\Pi_{r,G'}[X \to A'\cap X]|$ for each $X \in \mathcal{X}$. 

\medskip
Now consider some set $X \in \mathcal{X}$. 
We will show that $|\Pi_{r,G'}[X \to A'\cap X]| \leq 2\times 49^3 (r+1)^5 |A' \cap X|$. 

{\bf Case 1:} $|A'\cap X| \leq r+1$.

We apply directly Corollary~\ref{corollary:bound_profiles_outer_face} and obtain
\[
|\Pi_{r,G'}[X \to A'\cap X]| \leq (r+1)^3|A'\cap X|^3 \leq (r+1)^5|A'\cap X|.
\]

\medskip
{\bf Case 2:} $|A'\cap X| > r+1$.

Consider the cycle $C'$ as being oriented clockwise. 
Let $a_1, \dots, a_\ell$ denote the vertices in  $A'\cap X$ according to their order on $C'$. 
Let us mark each vertex $a_i$ with $i\in [\ell]$ such that $i=k(r+1)+1$ for some integer $k \geq 0$.  
Note that there are at least two marked vertices, since $\ell=|A'\cap X| > r+1$.  
We say that two marked vertices are {\em consecutive} if they are consecutive among marked vertices in the clockwise order around $C'$. 
As $G'[X]$ has radius at most $24r$, there exists a vertex $s \in X$
that is at distance at most $24r$ in $G'[X]$ from every vertex in $X$. 
By considering the union of a shortest $(a,s)$-path for every marked vertex $a$, one can find a tree $T$ in $G'$ rooted at $s$ containing all marked vertices and having height at most $24r$; see Figure~\ref{fig:planar_NC}. 
(The {\em height} of a rooted tree is defined as the maximum length of a path from a root to a leaf in the tree.)  

\begin{figure}[ht]
    \centering
    \begin{tikzpicture}   
        \draw (0,0) circle (5);
        \node[vertexverysmall, red, top color=red, bottom color=red] (s) at (0,0) {};
        
        \foreach \i in {0,...,14}{
          \node[vertexverysmall] (a\i) at (\i*360/15:5) {};
        }
        
        \node[vertexverysmall, red, top color=red, bottom color=red] (a1) at (1*360/15:5) {};
        \node at ($(1*360/15:5) + (60:.5)$) {$a_1$};
        
        \node[vertexverysmall, red, top color=red, bottom color=red] (a4) at (4*360/15:5) {};
        \node at ($(4*360/15:5) + (90:.5)$) {$a_{r+2}$};
        
        \node[vertexverysmall, red, top color=red, bottom color=red] (a8) at (7*360/15:5) {};
        \node at ($(7*360/15:5) + (180:.75)$) {$a_{2r+3}$};
        
        \node[vertexverysmall, red, top color=red, bottom color=red] (a8) at (10*360/15:5) {};
        \node at ($(10*360/15:5) + (-120:.5)$) {$a_{3r+4}$};
        
        \node[vertexverysmall, red, top color=red, bottom color=red] (a8) at (13*360/15:5) {};
        \node at ($(13*360/15:5) + (-45:.55)$) {$a_{4r+5}$};
        
        \node at (2.5*360/15:3) {$V_{a_1,a_{r+2}}$};
        \node at (5.5*360/15:3) {$V_{a_{r+2},a_{2r+3}}$};
        \node at (8.5*360/15:3) {$V_{a_{2r+3},a_{3r+4}}$};
        \node at (11.5*360/15:3) {$V_{a_{3r+4},a_{4r+5}}$};
        \node at (13.5*360/15:3) {$V_{a_{4r+5},a_{1}}$};
        
        \draw[red, very thick, dashed] (0,0) -- (0:2) -- (a1);
        \draw[red, very thick, dashed] (0:0) -- (-4.5*360/15:1.5) -- (a13);
        \draw[red, very thick, dashed] (-4.5*360/15:1.5) -- (a10);
        \draw[red, very thick, dashed] (0:0) -- (5.5*360/15:1.5) -- (a7);
        \draw[red, very thick, dashed] (5.5*360/15:1.5) -- (a4);
        
        \node at (-160:0.5) {$s$};
        \node[red] at (15:1.8) {$T$};
        
        \end{tikzpicture}
    \caption{The main step in the proof of Theorem~\ref{theorem:NC_planar_degree_6}.}
    \label{fig:planar_NC}
\end{figure}
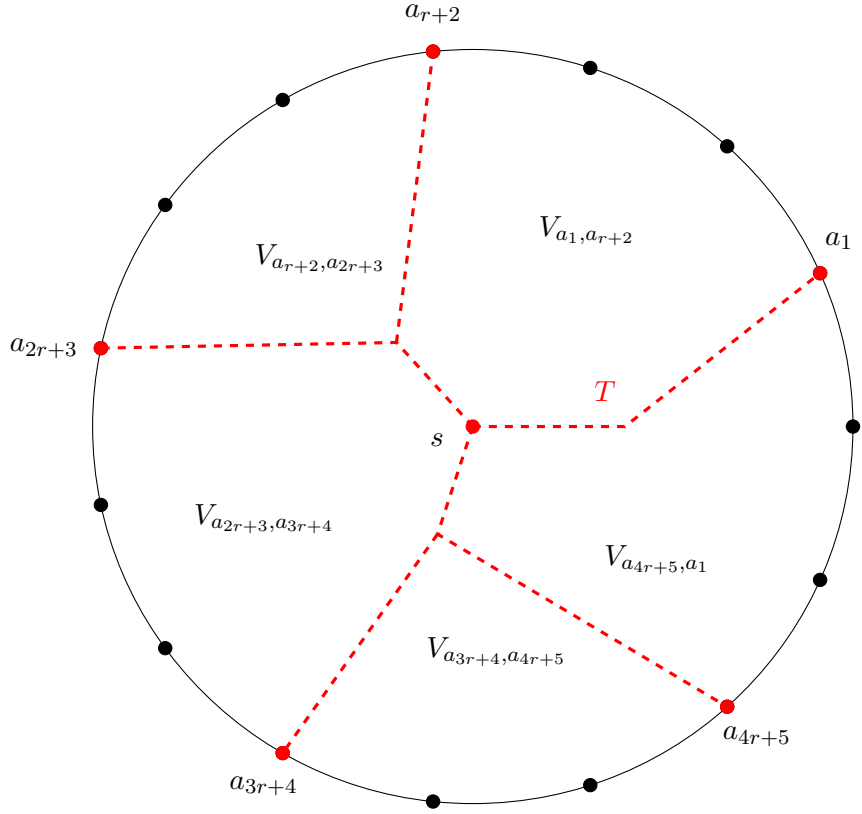

Consider a pair $a, b$ of consecutive marked vertices. 
Then there is a (unique) path $P_{a, b}$ linking $a$ and $b$ in $T$.   
Furthermore, $P_{a, b}$ contains at most $48r+1$ vertices. 
Let $V_{a, b}$ be the set of vertices of $G'$ enclosed by the cycle $B_{a, b}$ defined as the union of $P_{a, b}$ and the path from $a$ to $b$ on the (oriented) cycle $C'$, including vertices in the cycle $B_{a, b}$ itself.  
We also denote by $G'_{a, b}$ the subgraph of $G'$ with vertex set $V_{a, b}$ and edge set all the edges drawn in the closed disc bounded by $B_{a, b}$.
In other words, $G'_{a, b}$ is obtained from $G'[V_{a, b}]$ by removing the chords of $B_{a, b}$ drawn outside $B_{a, b}$ in $G'$. 
Let $A'_{a, b}$ be the union of $V(P_{a, b})$ and   
all vertices in $A'\cap X$ on the $(a,b)$-path in $C'$ going from $a$ to $b$ on $C'$ in clockwise order. 
Observe that $|A'_{a, b}| \leq 49r+1$ and that all vertices in $A'_{a, b}$ lie on the outerface of $G'_{a, b}$.    
Thus, by Corollary~\ref{corollary:bound_profiles_outer_face},
$|\Pi_{r,G'_{a, b}}[V_{a, b}\to A'_{a, b}]| \leq (r+1)^3(49r+1)^3$.

Now, let $Q$ be the set of pairs $(a,b)$ of consecutive marked vertices. 
Note that $|Q| = \left\lceil \frac{|A'\cap X|}{r+1}\right\rceil \leq 2 \frac{|A'\cap X|}{r+1}$. 
Hence, 
\[
\begin{split}
|\Pi_{r,G'[X]}[X \to A'\cap X]| 
&\leq \sum_{(a,b)\in Q}|\Pi_{r,G'_{a, b}}[V_{a, b} \to A'_{a, b}]| \\
&\leq |Q| (r+1)^3(49r+1)^3 \\
&\leq 2 \times 49^3 (r+1)^5 |A' \cap X|. \\
\end{split}
\]

\medskip

Thus, in both cases $|\Pi_{r,G'[X]}[X \to A'\cap X]| \leq 2\times 49^3 (r+1)^5 |A' \cap X|$ holds for every $X \in \mathcal{X}$, as claimed. 

It follows that 
\[
\def\arraystretch{1.5}
\begin{array}{r l l}
|\Pi_{r,G}[V(G) \to A]| 
&\leq |\Pi_{r,G'}[V(G') \to A']| & \text{by (iv)} \\
&\leq \sum_{X \in \mathcal{X}} |\Pi_{r,G'[X]}[X \to A' \cap X]| & \text{by Lemma~\ref{lemma:covers_and_cuts}} \\
&\leq \sum_{X \in \mathcal{X}} 2 \times 49^3 (r+1)^5 |A' \cap X| & \\
&\leq 2 \times 49^3 (r+1)^5 18|A'| & \\
&\leq 49^3 \times 36 (r+1)^5 2(2r+1)|A| & \text{by (i)}\\
&\leq 10^8 (r+1)^6 |A|
\end{array}
\]
which concludes the proof of the theorem. 
(Note that we made no attempt to optimize the constant $10^8$ in the proof.)
\end{proof}

\subsection{Third proof}

We will now improve the bound in Theorem~\ref{theorem:NC_planar_degree_6} from $\bigO(r^6 |A|)$ to $\bigO(r^4 |A|)$ using the following theorem of Li and Parter~\cite{li_parter_stoc_2019}. 

\begin{theorem}[Li and Parter~\cite{li_parter_stoc_2019}]\label{theorem:base_bound_Li_Parter}
    Let $G$ be a plane graph, let $A$ be a nonempty set of consecutive vertices on the boundary of a face of $G$, and let $r$ be a nonnegative integer. 
    Then $|\Pi_{G,r}[V(G) \to A]| \in \bigO(r|A|^3)$. 
\end{theorem}

A {\em near-triangulation} is a plane graph whose outerface is bounded by a cycle of the graph, and all of whose inner faces are bounded by triangles of the graph. 
We will use the following well-known variation on a lemma of Sperner, see e.g. \cite{AZ18}.  

\begin{lemma}[Sperner's Lemma]\label{lemma:sperner}
    Let $G$ be a near-triangulation and suppose that the cycle bounding its outerface is partitioned into three non empty paths $P_1,P_2,P_3$.
    Let $\phi: V(G) \to \{1,2,3\}$ be a coloring of the vertices of $G$ such that $\phi(u)=i$ for every $u \in V(P_i),i=1,2,3$.
    Then there is an inner face bounded by a triangle $xyz$ in $G$ with $\phi(x)=1,\phi(y)=2,\phi(z)=3$.
\end{lemma}

Given a vertex subset $R$ of a graph $G$ and a rooted forest $F$ in $G$, we say that {\em $F$ is rooted in $R$} if $R$ coincides with the set of roots of the trees in $F$. 
The {\em height} of a rooted forest $F$ is the maximum length of a path from a root to a leaf in $F$. 
The proof strategy for the following lemma is adapted from ideas in~\cite{PS21, dujmovic_planar_2020}, in particular it relies on a variant of the so-called `tripod decomposition' from~\cite{ dujmovic_planar_2020}. 

\begin{lemma}\label{lemma:structural_lemma_best_bound_planar}
Let $r$ be a positive integer. 
Let $G$ be a connected plane graph whose outerface is bounded by a cycle $C$. 
Let $\widetilde{G}$ be a plane near-triangulation with $V(\widetilde{G})=V(G)$ obtained from $G$ by triangulating every inner face of $G$ arbitrarily. 
Suppose that $G$ has a spanning forest $F$ rooted in $V(C)$ whose height is at most $r$. 
Then, there exists a subgraph $H$ of $\widetilde{G}$ satisfying: 
\begin{enumerate}
    \item $C$ is contained in $H$; 
    \item $H$ has at most $3|V(C)|/r+1$ inner faces, and
    \item \label{prop:inner_face_lemma}
    every inner face of $H$ is bounded by a cycle of $H$ of length at most $6r$ containing at most three edges not in $G$. 
\end{enumerate}
\end{lemma}

\begin{proof}
    We start by introducing some terminology used in the proof.  
    A {\em vertical path} in $F$ is a subpath of a path between a vertex and a root in $F$. Vertical paths are naturally oriented towards the root; the {\em source} and {\em sink} of a vertical path are defined according to this orientation.  
    A \emph{tripod} $Y$ is the union of three pairwise disjoint vertical paths $P_1,P_2,P_3$ whose sources form a triangle $t$ bounding a face in $\widetilde{G}$.
    We call $P_1,P_2,P_3$ the \emph{branches} of $Y$, $t$ the \emph{center} of $Y$, and the sinks of $P_1,P_2,P_3$ the \emph{feet} of $Y$. 
    The path obtained by taking the union of two branches and the edge of the center linking their sources is called a {\em side} of the tripod. 
    (Thus, every tripod has three sides.) 
    
    We maintain a collection $\mathcal{Y}$ of tripods, a distinguished foot $f_Y$ of $Y$ for every tripod $Y \in \mathcal{Y}$, and a subgraph $H$ of $\widetilde{G}$ satisfying the following properties: 
    \begin{enumerate}
        \item \label{prop:1}
        every face of $H$ is bounded by a cycle of $H$ (i.e.\ $H$ is $2$-connected);
        \item \label{prop:2} $C$ is contained in $H$, and thus $C$ bounds the outerface of $H$;
        \item \label{prop:3} $V(H) = V(C) \cup \bigcup_{Y \in \mathcal{Y}} V(Y)$; 
        \item \label{prop:4} for every inner face $f$ of $H$, the cycle $C_f$ of $H$ bounding $f$ satisfies at least one of the following three properties: 
        \begin{enumerate}
        \item \label{prop:a} $C_f$ has length at  most $6r$ and contains at most two edges in $E(\widetilde{G}) \setminus E(G)$, 
        \item \label{prop:b} the intersection of $C_f$ with $C$ is a path, and moreover 
        $C_f$ is the union of two internally disjoint paths $P_f, Q_f$ where $P_f = C_f\cap C$ and $Q_f$ is a side of some tripod in $\mathcal{Y}$,
        \item \label{prop:c} $C_f$ is the center of a tripod $Y$ in $\mathcal{Y}$,
        \end{enumerate}
        \item \label{prop:5} the number of inner faces of $H$ is at most $3|\mathcal{Y}|+1$, and
        \item \label{prop:6}
        for every tripod $Y \in \mathcal{Y}$, $f_Y$ belongs to $C$, and the elements of $(f_Y)_{Y \in \mathcal{Y}}$ are pairwise at distance at least $r$ on $C$.
    \end{enumerate}
    For the reader familiar with the tripod decomposition from~\cite{ dujmovic_planar_2020}, let us emphasize that here the tripods in $\mathcal{Y}$ are allowed to intersect, and they always have three branches. 

    First, we show that objects $\mathcal{Y}, (f_Y)_{Y \in \mathcal{Y}},H$ satisfying properties \ref{prop:1}--\ref{prop:6} do exist. 
    This can be seen as follows.     
    Partition $C$ into three non-empty paths $P_1, P_2, P_3$, and define a coloring $\phi$ of $V(\widetilde{G})$ as follows:
    \begin{itemize}
        \item for every $i=1,2,3$, for every $v \in V(P_i)$, let $\phi(v)=i$, and
        \item every vertex of $\widetilde{G}$ not in $C$ receives the color of its parent in $F$.
    \end{itemize} 
    Observe that the coloring is well defined, since all roots of $F$ are colored in the first step.      
    By Lemma~\ref{lemma:sperner}, there is a triangle $uvw$ in $\widetilde{G}$ such that $\phi(u)=1, \phi(v)=2$ and $\phi(w)=3$.
    Let $Y$ be the tripod with center $uvw$ and whose branches are the three vertical paths starting at respectively $u,v,w$ and ending on $C$ (including their endpoints on $C$), 
    and let $f_Y$ be the sink of one (arbitrarily chosen) branch of $Y$.
    Then the triple $\mathcal{Y} = \{Y\}, (f_Y), H=C \cup Y$ satisfies properties \ref{prop:1}--\ref{prop:6}. 
    Then the triple $\mathcal{Y} = \{Y\}, (f_Y), H=C \cup Y$ satisfies properties \ref{prop:1}--\ref{prop:6}:
    \ref{prop:1}, \ref{prop:2}, \ref{prop:3}, \ref{prop:5} and \ref{prop:6} are clear from the definition, and for every inner face $f$ of $H$, either \ref{prop:b} or \ref{prop:c} holds.

    We now consider a choice for $\mathcal{Y}, (f_Y)_{Y \in \mathcal{Y}},H$ satisfying properties \ref{prop:1}--\ref{prop:6} and maximizing $|\mathcal{Y}|$. 
    We claim that every inner face $f$ of $H$ satisfies property~\ref{prop:inner_face_lemma} of the lemma, that is: 

\begin{itemize}
    \item[(*)] $f$ is bounded by a cycle $C_f$ of $H$ of length at most $6r$ and $C_f$ contains at most three edges not in $G$. 
\end{itemize}
    To see this, suppose for a contradiction that this does not hold for some inner face $f$ of $H$.     
    Then the cycle $C_f$ does not satisfy properties~\ref{prop:a} and~\ref{prop:c}, and thus it must satisfy property~\ref{prop:b}, that is,  $C_f$ is the union of two internally disjoint paths $P_f, Q_f$ where $P_f = C_f\cap C$ and $Q_f$ is a side of some tripod in $Y^* \in \mathcal{Y}$. 
    Since $F$ is a subgraph of $G$,
    this implies that $C_f$ contains at most {\em one} edge not in $G$, namely the unique edge of the center of $Y^*$ in $C_f$ (if this edge is not in $G$).
    Since $C_f$ does not satisfy property~\ref{prop:a}, it follows that $C_f$ has length at least $6r+1$. 
    Observe that $Q_f$ has length at most $2r+1$, since the rooted forest $F$ has height at most $r$. 
    Thus, $P_f$ has length at least $6r+1 - (2r+1) = 4r$. 
    
    Let $\widetilde{G}_f$ be the subgraph of $\widetilde{G}$ consisting of all the vertices and edges of $\widetilde{G}$ that are in the closed disc defined by $C_f$.     
    Let $Q_{f,1}$ and $Q_{f,2}$ be the two branches of $Y^*$ that are contained in $C_f$. 
    We orient $C_f$ clockwise. 
    Starting from $P_f$ and going along $C_f$ in this order, we may assume that we meet $Q_{f,1}$ before $Q_{f,2}$. Note that the orientation of $C_f$ induces an orientation of $P_f$, which we use below.

    Let $\phi:V(\widetilde{G}_f) \to \{1,2,3\}$ be the coloring of the vertices of $\widetilde{G}_f$ defined as follows:
    \begin{itemize}
        \item vertices in $Q_{f,1}$ and the last $r$  vertices of $P_f$ are colored $1$;
        \item vertices in $Q_{f,2}$ and the first $r$ vertices of $P_f$ are colored $2$; 
        \item the remaining uncolored vertices of $P_f$ are colored $3$, and
        \item every vertex in the proper interior of $f$ receives the color of its parent in $F$.
    \end{itemize}
    Note that the coloring of a vertex $v$ considered in the last item above is defined inductively, on the distance from $v$ to the root $r$ of $F$ of the corresponding tree in $F$. It is well defined since $r$ is not in the proper interior of $f$ and thus the vertical path from $v$ to $r$ in $F$ intersects $C_f$, which is already colored.

    Observe that there are vertices colored $3$ on $C_f$ since $|V(P_f)| \geq 4r \geq 2r+1$. 
    Thus, by Lemma~\ref{lemma:sperner}, there exists a triangle $uvw$ in $\widetilde{G}_f$ with $\phi(u)=1,\phi(v)=2,\phi(w)=3$.
    Let $Y'$ be the tripod with center $uvw$ and whose three branches are the three vertical paths starting at respectively $u,v,w$ and ending in $C_f$ (including their endpoints in $C_f$), see Figure~\ref{fig:best_bound_planar}.
    \begin{figure}[ht]
        \centering
        \includegraphics{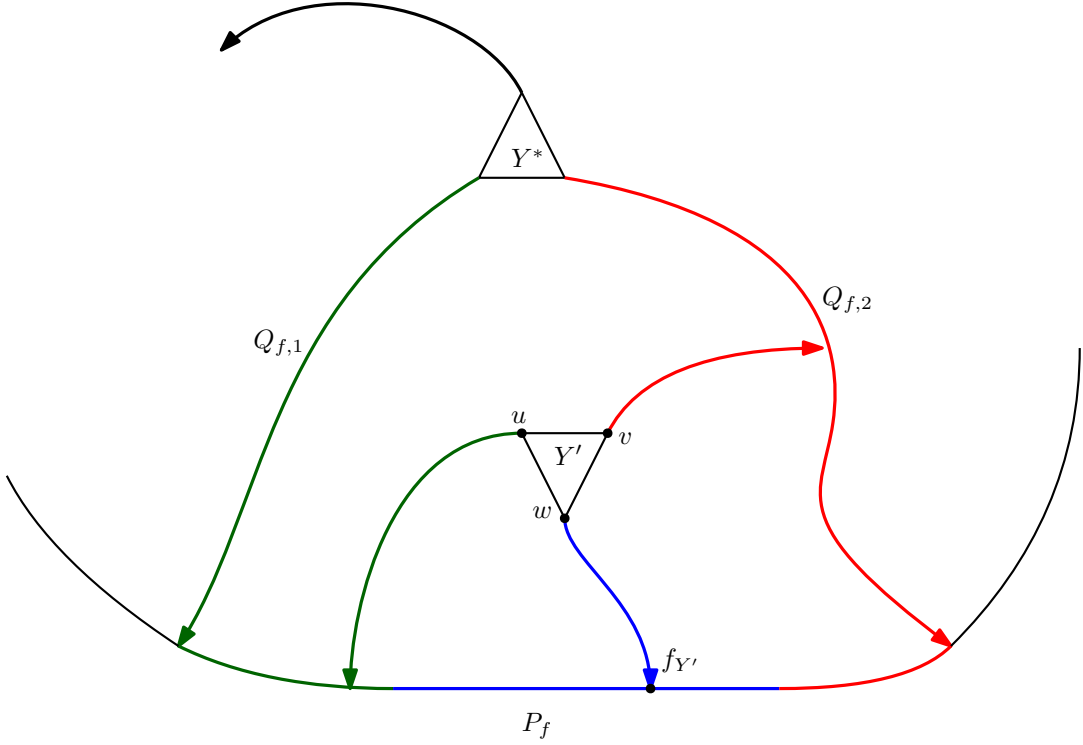}
        \caption{Illustration of the proof of Lemma~\ref{lemma:structural_lemma_best_bound_planar}.}
        \label{fig:best_bound_planar}
    \end{figure}

    We now define $\mathcal{Y'}:=\mathcal{Y} \cup \{Y'\}$, we let $f_{Y'}$ be the sink of the branch of $Y'$ starting at $w$, and let $H'$ be obtained from $H$ by adding the tripod $Y'$ to $H$ (that is, $H':=H \cup Y'$).
    We now check that $\mathcal{Y'},(f_Y)_{Y \in \mathcal{Y'}},H'$ satisfy properties \ref{prop:1}--\ref{prop:6}:
    \begin{enumerate}
        \item Every face of $H'$ is bounded by a cycle of $H'$. 
        \item $C$ is contained in $H'$.
        \item $V(H') = V(C) \cup \bigcup_{Y \in \mathcal{Y'}} V(Y)$.  
        \item The three inner faces of $H'$ that were not faces of $H$ were created by adding the tripod $Y'$ inside the face $f$ of $H$.
        First the face bounded by the  center of $Y'$ satisfies \ref{prop:c}. We now consider the other created faces.
        Each of these inner faces is bounded by a cycle of $H'$.
        Observe that if such a cycle has a non-empty intersection with $C$ then this intersection is a path contained in $C$. 
        Now, among these three cycles, the cycle that corresponds to colors $1$ and $2$ contains at most $6r$ vertices, since there are at most $3r$ vertices colored $1$ ($2r$ on $C_f$ plus $r$ from $Y'$) and the same for color $2$. 
        This cycle contains at most two edges that are possibly not in $G$, namely one edge from the center of $Y^*$ and the edge $uv$ from the center of $Y'$. 
        Thus, this cycle satisfies property~\ref{prop:a}. 
        By construction, the other two cycles corresponding to the pairs of colors $(1,3)$ and $(2,3)$ respectively, satisfy property~\ref{prop:b}, since each such cycle is the union of a path $P$ contained in $C_f$ with a side of $Y'$, and $P$ intersects $P_f$ in a subpath of $P_f$. 
        \item The number of faces of $H'$ is at most $3|\mathcal{Y}|+1 + 4-1 = 3|\mathcal{Y'}|+1$.
        \item In order to verify this property, one only needs to check $f_{Y'}$: It is in $C$ and moreover it is at distance at least $r$ in $C$ from all the vertices in $\{f_{Y}\}_{Y \in \mathcal{Y}}$ since $f_{Y'}$ was colored with color $3$. (Indeed, observe that every path in $C$ from $f_{Y'}$ to some $f_Y$ with $Y \in \mathcal{Y}$ either contains the  last $r$ vertices of $P_f$, colored $1$, or the first $r$ vertices of $P_f$, colored $2$.) 
    \end{enumerate}
    However, this shows that the triple $\mathcal{Y'},(f_Y)_{Y \in \mathcal{Y'}},H'$ is a better choice than $\mathcal{Y}, (f_Y)_{Y \in \mathcal{Y}},H$, which contradicts the maximality of $|\mathcal{Y}|$. 
    
    We deduce that every inner face $f$ of $H$ satisfies property~\ref{prop:inner_face_lemma} of the lemma. 
    Moreover, it follows from property~\ref{prop:6} that $|\mathcal{Y}| \leq |V(C)|/r$, implying that $H$ has at most $3|\mathcal{Y}|+1 \leq 3|V(C)|/r+1$ inner faces. 
    Therefore, $H$ satisfies all the required properties from the lemma. 
\end{proof}

We turn to the $\bigO(r^4)$ bound on the profile complexity of planar graphs, Theorem~\ref{thm:neighborhood_complexity_planar}. 
It follows from the following theorem. 

\begin{theorem}\label{theorem:NC_planar_degree_4}
    Let $G$ be a planar graph, let $A$ be a nonempty set of vertices of $G$, and let $r$ be a nonnegative integer. 
    Then, $|\Pi_{r,G}[V(G) \to A]| \in \bigO(r^4|A|)$.
\end{theorem}

\begin{proof}
The theorem trivially holds for $r=0$, so we may assume $r \geq 1$ (which is needed below when using Lemma~\ref{lemma:structural_lemma_best_bound_planar}). 
If there is a vertex at distance at least $r+1$ from $A$ in $G$, then we can remove it without decreasing $|\Pi_{r,G}[V(G) \to A]|$.
Thus, we may assume that every vertex of $G$ is at distance at most $r$ from $A$. 
By Lemma~\ref{lemma:reduction_outerface}, there is a plane graph $G'$ whose outerface is bounded by a cycle such that $|\Pi_{r,G}[V(G) \to A]| \leq |\Pi_{r,G'}[V(G') \to A']|$ and $|A'| \leq 2(2r+1)|A|$, where $A'$ is the set of vertices on the outerface of $G'$. 

Our goal is to bound $|\Pi_{r,G'}[V(G') \to A']|$ from above. 
Again, we may assume that every vertex in $G'$ is at distance at most $r$ from $A'$, and thus there is a spanning rooted forest $F$ in $G'$ whose set of roots coincide with $A'$ and with height at most $r$. 
By Lemma~\ref{lemma:structural_lemma_best_bound_planar}, there is a subgraph $H$ of a near triangulation $\widetilde{G}$ obtained from $G'$ by triangulating every inner face of $G'$ such that $H$ has at most $3|A'|/r+1$ inner faces, 
and every inner face of $H$ is bounded by a cycle of length at most $6r$ and contains at most $3$ edges in $E(\widetilde{G})\setminus E(G')$. 

Let $f_1, f_2, \dots, f_q$ be the inner faces of $H$, and for every $i\in [q]$ let  $C_{f_i}$ denote the cycle of $H$ bounding $C_{f_i}$. 
For every $i\in [q]$, let $\widehat{G}_i$ be the union of the cycle $C_{f_i}$ with the subgraph of $G'$ drawn inside the face $f_i$ of $H$.  
(Note that $\widehat{G}_i$ is not necessarily a subgraph of $G'$, since $C_{f_i}$ could contain edges not in $G'$; also, note that in $G'$ the cycle  $C_{f_i}$ could have chords drawn outside $f_i$, which are then not included in $\widehat{G}_i$.)  
We denote by $V_i$ the vertex set of $\widehat{G}_i$.
Also, we let $E_i$ be the set of edges in $C_{f_i}$ that are not in $E(G')$. 
Thus, $|E_i| \leq 3$. 

Let $i\in [q]$. 
For each edge $uv\in E_i$ (if any), replace in $\widehat{G}_i$ the edge $uv$ by a path $P_i$ between $u$ and $v$ of length $r+1$, where the inner vertices are new vertices, and drawn where the edge $uv$ was drawn in $\widetilde{G}$. 
This results in a plane graph $G_i$ whose outerface is bounded by a cycle $C_i$ of length at most $9r$. 
Moreover, one can check that 
for every $v \in V_i$ and every $c \in V(C_{f_i})$, if $\dist_{G'[V_i]}(v,c) \leq r$, then
\[
\dist_{G'[V_i]}(v,c) = \min_{c'\in V(C_{f_i})} (\dist_{G_i}(v,c')+\dist_{G'[V_i]}(c',c)).
\]
(We remark that we do not simply have $\dist_{G'[V_i]}(v,c) = \dist_{G_i}(v,c)$ because in $G'[V_i]$ some chords of $C_{f_i}$ could be drawn outside $f_i$, and those chords are not included in $G_i$.) 
Moreover, if $\dist_{G'[V_i]}(v,c) > r$, then $\dist_{G_i}(v,c) > r$.
Hence
\[
\pi_{r,G'[V_i]}(v,c) = \mathrm{Cap}_r \circ \min_{c'\in V(C_{f_i})} (\pi_{r,G_i}(v,c')+\dist_{G'[V_i]}(c',c)).
\]
We deduce that 
\[
|\Pi_{r,G'[V_i]}[V_i \to V(C_{f_i})]| 
\leq |\Pi_{r,G_i}[V(G_i) \to V(C_{f_i})]| 
\leq |\Pi_{r,G_i}[V(G_i) \to V(C_i)]| 
\in \bigO(r^4), 
\] 
where the rightmost bound comes from  Lemma~\ref{theorem:base_bound_Li_Parter} applied on the pair $G_i,V(C_i)$.  

Applying Lemma~\ref{lemma:covers_and_cuts} with $S_i=V(C_{f_i})$ for every $i\in [q]$, we conclude that
\[
|\Pi_{r,G'}[V(G') \to A]| \leq \sum_{i= 1}^q |\Pi_{r,G[V_i]}[V_i \to V(C_{f_i})]| \in \bigO(r^4 |A'|/r)
\]
and we conclude that $|\Pi_{r,G}[V(G) \to A]| \leq |\Pi_{r,G'}[V(G') \to A'] \in \bigO(r^4|A|)$.
This concludes the proof of the theorem.
\end{proof}

\section{Graphs on surfaces}
\label{section:surfaces}

In this section, we bound the profile complexity of graphs on surfaces. 
We start by recalling some definitions and results about graphs on surfaces. 
We assume some familiarity of the reader with this topic,  see the textbook of Mohar and Thomassen~\cite{mohar_thomassen_graphs_on_surfaces} for background on this topic, and for undefined terms. 
The {\em Euler genus} of the sphere with $g$ handles added is $2g$, and that of the sphere with $g$ crosscaps added is $g$. 
The {\em Euler genus} of a graph $G$ is the minimum Euler genus of a surface in which $G$ can be embedded into. 
A graph $G$ is {\em cellularly embedded} in a surface if every face of the embedding is homeomorphic to an open disk. 
A cycle $C$ in a graph embedded in a surface $\Sigma$ is said to be {\em noncontractible} if $C$ is noncontractible in  $\Sigma$ when seen as a closed curve. 
The cycle $C$ is {\em separating} if removing $C$ from $\Sigma$ separates $\Sigma$ in two pieces, and {\em nonseparating} otherwise. 
The cycle $C$ is {\em $1$-sided} if its neighborhood on the surface is homeomorphic to the Möbius band, and {\em $2$-sided} otherwise, in which case its neighborhood is homeomorphic to the cylinder. 

We need two lemmas about the standard operation of `cutting along' a cycle in a graph embedded in a surface. We refer the reader to~\cite[Section~4.2]{mohar_thomassen_graphs_on_surfaces} for the definition of this operation. 
Here, we limit ourselves to recalling informally the three different situations than can occur when cutting along a cycle $C$: (1) If $C$ is separating, then cutting along $C$ produces two graphs, each having a copy of the cycle $C$. 
(2) If $C$ is nonseparating and $2$-sided, then cutting along $C$ produces one graph, with two vertex-disjoint cycles $C_1$ and $C_2$ that are copies of $C$. 
(3) If $C$ is nonseparating and $1$-sided, then cutting along $C$ produces one graph, with one cycle $C_1$ corresponding to $C$ but having twice the length of $C$. 

\begin{lemma}[{\cite[Proposition 4.2.1]{mohar_thomassen_graphs_on_surfaces}}]
\label{lemma:separating_cycle}
If $G$ is a graph cellularly embedded in a surface of Euler genus $g$, and $C$ is a noncontractible separating cycle $C$, then cutting along $C$ gives two graphs $G_1$ and $G_2$ of Euler genera $g_1$ and $g_2$, respectively, such that $g_1< g$, $g_2< g$, and $g_1 + g_2=g$. 
\end{lemma}

\begin{lemma}[{\cite[Proposition 4.2.4]{mohar_thomassen_graphs_on_surfaces}}]
\label{lemma:nonseparating_cycle}
If $G$ is a graph cellularly embedded in a surface of Euler genus $g$, and $C$ is a noncontractible nonseparating cycle $C$, then cutting along $C$ gives one graph $\Tilde{G}$ with Euler genus $\Tilde{g} < g$. 
\end{lemma}

A {\em geodesic} in a graph $G$ is a path in $G$ that is shortest among all paths connecting its two endpoints. 
We will need the following well-known lemma, which follows from Proposition 4.3.1(a) in \cite{mohar_thomassen_graphs_on_surfaces}. 

\begin{lemma}[{\cite[Proposition 4.3.1(a)]{mohar_thomassen_graphs_on_surfaces}}]
\label{lemma:noncontractible_cycle_formed_of_geodesics}
If $G$ is a graph cellularly embedded in a surface of Euler genus $g>0$, then there exists
a noncontractible cycle $C$ that is the union of two geodesics in $G$ having the same endpoints. 
\end{lemma}

We may now bound the number of profiles for graphs embedded in a fixed surface.

\begin{theorem}\label{theorem:bounded_genus}
Let $h:\mathbb{N} \to \mathbb{N}$ be a function such that $|\Pi_{r,G}[V(G) \to A]| \leq h(r)|A|$ for every planar graph $G$, every nonempty set $A$ of vertices of $G$, and every nonnegative integer $r$. 
Then, there exists a function $f:\mathbb{N} \to \mathbb{N}$ such that for every nonnegative integer $g$, every graph $G$ of Euler genus $g$, every nonempty set $A$ of vertices of $G$, and every nonnegative integer $r$,  
\[
|\Pi_{r,G}[V(G) \to A]| \leq f(g)h(r)(r+|A|);
\]
in particular, 
\[
|\Pi_{r,G}[V(G) \to A]| \in \bigO_g(r^4(r+|A|))
\]
using the bound $h(r)\in \bigO(r^4)$ from Theorem~\ref{theorem:NC_planar_degree_4}. 
\end{theorem}

The proof works as follows: 
While there exists a noncontractible cycle of size $\bigO(r)$, we cut it,
which decreases the genus. Finally, when there is no more such short noncontractible
cycle, the resulting graph is locally planar, and we can apply Theorem~\ref{theorem:NC_planar_degree_4}.

\begin{proof}[Proof of Theorem~\ref{theorem:bounded_genus}]
Graphs of Euler genus $g$ forbid $K_t$ as a minor for some $t \in \bigO(\sqrt{g})$ (see e.g.\ \cite{mohar_thomassen_graphs_on_surfaces}). 
By Theorem~\ref{theorem:construct_X_minor_closed}, there are thus constants $c_g, c'_g$ with $c_g \in \bigO(g)$ such that, for every nonnegative integer $r$, graphs of  Euler genus $g$ have $(r, c_g r, c'_g)$-sparse covers.  
These sparse covers will be used in Case~2 of the proof below. 

We prove the theorem with the function $f(g)$ defined as follows: 
\begin{align*}
f(0) &:=1 \\
f(g) &:= \max\{12(c_g + 1)f(g-1), c'_g\}  \quad \quad \forall g \geq 1. 
\end{align*}

Let $G$ be a graph of Euler genus $g$ and consider an embedding of $G$ in a surface $\Sigma$ of Euler genus $g$. 
We prove the theorem by induction on the Euler genus $g$. 
We may assume $g>0$, since the case $g=0$ follows from Theorem~\ref{theorem:NC_planar_degree_4}. 
Note that $G$ is cellularly embedded in $\Sigma$ (otherwise, one could embed $G$ in a surface of smaller Euler genus).  

\medskip

{\bf Case 1 (induction step): There exists a noncontractible cycle $C$ of length at most $4(c_g+1)r$ in $G$.}
First suppose that $C$ is separating. 
Let $G_1$ and $G_2$ be the two graphs resulting from cutting along $C$, and for $i=1,2$ let $C_i$ denote the copy of $C$ in $G_i$, and let $g_i$ be the Euler genus of $G_i$. 
We also denote by $A_i$ the set $A \cap V(G_i)$ for every $i=1,2$.
Then $g_1<g$ and $g_2<g$ with $g_1+g_2=g$, by Lemma~\ref{lemma:separating_cycle}. 
By Lemma~\ref{lemma:covers_and_cuts} and the induction hypothesis
\[
\begin{split}
|\Pi_{r,G}[V(G) \to A]| 
&\leq |\Pi_{r,G_1}[V(G_1) \to A_1 \cup V(C_1)]| + |\Pi_{r,G_2}[V(G_2) \to A_2 \cup V(C_2)] \\
&\leq f(g-1)h(r)(8(c_g+1)r+|A|) \\
&\leq f(g)h(r)(r+|A|),
\end{split}
\]
as desired. 

Next, assume that $C$ is nonseparating.  
Let $\Tilde{G}$ be the graph obtained by cutting along $C$. 
If $C$ is $2$-sided, let $C_1$ and $C_2$ be the two resulting copies of $C$ in $\Tilde{G}$. 
If $C$ is $1$-sided, let $C_1$ be the cycle in $\Tilde{G}$ corresponding to $C$ (which thus has twice the length of $C$), and define $C_2$ as being empty (to help uniformize the discussion below). 
By Lemma~\ref{lemma:nonseparating_cycle}, $\Tilde{G}$ has Euler genus at most $g-1$. 

For every vertex $w \in V(C)$, we denote by $w_1,w_2$ the two corresponding vertices in $\Tilde{G}$ (one in $C_1$ and another in $C_2$ if $C$ is $2$-sided, or the two in $C_1$ if $C$ is $1$-sided). 

Let $A':=(A\setminus V(C)) \cup V(C_1) \cup V(C_2)$. 
We prove the following:
\begin{equation}\label{eq:boundG'}
|\Pi_{r,G}[V(G) \to A]| \leq |\Pi_{r,\Tilde{G}}[V(\Tilde{G}) \to  A']| + |V(C)|
\end{equation}
To see this, let $v \in V(G) \setminus V(C)$ and $a \in A$. 
If $a \not\in V(C)$ then
\[
\dist_{G}(v,a) = \min\left( 
\min_{w \in V(C)}
\min_{i=1,2}(\dist_{\Tilde{G}}(v,w_i) + \dist_{G}(w,a)),
\dist_{\Tilde{G}}(v,a)\right)
\]
and as a consequence
\[
\pi_{r,G}[v\to a] = \mathrm{Cap}_r \circ \min\left( \min_{w \in V(C)}
\min_{i=1,2}(\pi_{r,\Tilde{G}}[v \to w_i] +
\dist_{G}(w,a)), \pi_{r,\Tilde{G}}[v \to a]\right).
\]
Similarly, if $a \in V(C)$ then
\[
\pi_{r,G}[v\to a] = \mathrm{Cap}_r \circ 
\min_{w \in V(C)}
\min_{i=1,2}(\pi_{r,\Tilde{G}}[v \to w_i] + \dist_{G}(w,a)).
\]
Finally, note that we trivially have $|\Pi_{r,G}[V(C) \to A]| \leq |V(C)|$. 

Combining the above observations, we obtain that $|\Pi_{r,G}[V(G) \to A]| \leq  |\Pi_{r,\Tilde{G}}[V(\Tilde{G}) \to A']| + |V(C)|$, which proves \eqref{eq:boundG'}.
Using the induction hypothesis on $\Tilde{G}$, we deduce that
\[
\begin{split}
|\Pi_{r,G}[V(G) \to A]|
&\leq |\Pi_{r,\Tilde{G}}[V(\Tilde{G}) \to A']| + |V(C)|  \\
&\leq f(g-1)h(r)(8(c_g+1)r+|A|) + 4(c_g+1)r \\
&\leq f(g)h(r)(r+|A|).
\end{split}
\]

\medskip

{\bf Case 2 (base case): Every noncontractible cycle in $G$ has length more than $4(c_g+1)r$.} 
Observe that in every induced subgraph $H$ of $G$ with radius at 
most $(c_g+1)r$, geodesics have length at most $2(c_g+1)r$, thus by
Lemma~\ref{lemma:noncontractible_cycle_formed_of_geodesics}, $H$ has no noncontractible
cycle, and hence $H$ is planar. 

Now apply Theorem~\ref{theorem:construct_X_minor_closed} to $G$ to obtain a $(r, c_g r, c'_g)$-sparse cover $\mathcal{X} \subseteq \mathcal{P}(V(G))$ of $G$.  

Let $X \in \mathcal{X}$ and let $v \in V(G)$ be such that every vertex of $X$ is at distance at most $c_g r$ from $v$ in $G[X]$. 
Observe that every vertex at distance at most $r$ from $A\cap X$ in $G$ is at distance at most $(c_g+1) r$ from $v$ in $G$. 
It follows that 
\[
\Pi_{r,G}[V(G) \to A\cap X] = \Pi_{r,G[N^{(c_g+1)r}[v]]}[N^{(c_g+1)r}[v] \to A\cap X].
\]
But as $G[N^{(c_g+1)r}[v]]$ is an induced subgraph of $G$ with radius at most $(c_g+1)r$,
it is planar. Thus, 
\[
|\Pi_{r,G}[V(G) \to  A\cap X]| \leq h(r)|A\cap X|.
\]
Summing up this inequality for all $X \in \mathcal{X}$, we obtain 
\[
\begin{split}
|\Pi_{r,G}[V(G) \to A]| 
&\leq \sum_{X \in \mathcal{X}, A \cap X \neq \emptyset} 
    h(r)|A \cap X| \\
&\leq h(r)\sum_{X \in \mathcal{X}, A \cap X \neq \emptyset}|A \cap X|\\
&\leq c'_g h(r)|A| \\
&\leq f(g)h(r)(r+|A|).
\end{split}
\]
This concludes the proof of Theorem~\ref{theorem:bounded_genus}.
\end{proof}

Using that $r^4(r+|A|) \leq r^5|A|$, we obtain the following corollary from Theorem~\ref{theorem:bounded_genus}. 

\begin{corollary}    \label{cor:neighborhood_complexity_surfaces} 
    For every nonnegative integer $g$, the profile complexity of graphs of Euler genus $g$ is in $\bigO_g(r^5)$. 
\end{corollary}

\section{Lower bounds}\label{section:lower_bounds}

So far we have seen that classes of graphs excluding a minor have polynomial
profile and neighborhood complexity. In this section, we give a lower bound on the degree
of this polynomial for graphs of treewidth at most $t$, and show that
$1$-planar graphs, which are some of the simplest graphs excluding no minor,
have super polynomial profile and neighborhood complexity.

\subsection{Bounded treewidth}

Let $t,r$ be positive integers with $r$ even. 
Consider the graph formed by an independent set $A = \{a_1, \dots, a_t\}$,
and for every $x=(x_1, \dots, x_t) \in [r/2,r]^t$, add a new vertex $v_x$
and link it to $a_i$ via a path of length $x_i$, for every $i \in [t]$.
One can easily check that $G$ has treewidth at most $t$ and that
the profile of $v_x$ on $A$ is $x$ for every $x \in [r/2,r]^t$.
This simple construction gives an example of a graph with treewidth $t$ and 
with a number of profiles that is at least $(r/2)^t = \bigOmega_t(r^t|A|)$.\footnote{Note that if one wants to have $|A| \gg t$ then it suffices to take many disjoint copies of this construction.}
The following theorem improves this lower bound to $\bigOmega_t(r^{t+1}|A|)$.

\begin{theorem}\label{theorem:construction_bounded_treewidth}
For every positive integers $t,r$ such that $2^t(t+1)$ divides $r$, there exists a graph $G$ of treewidth at most
$t$ and a set $A \subseteq V(G)$ such that
\[
|\Pi_{r,G}[V(G) \to A]| \geq \tfrac{1}{(2(t+1))^{t+1}} \cdot r^{t+1}|A| = \bigOmega_t(r^{t+1}|A|^{t+1}).
\]
\end{theorem}

Graphs of treewidth at most $2$ are planar. 
Hence, there are planar graphs whose numbers of profiles are at least $\bigOmega(r^3|A|)$. 
This is illustrated on Figure~\ref{fig:construction_bounded_treewidth}.

\begin{proof}[Proof of Theorem~\ref{theorem:construction_bounded_treewidth}] 
Let $\ell = \frac{r}{2(t+1)}$. We define a tree $T$ as follows:
\begin{itemize}
    \item $V(T) = [0,\ell-1]^{t+1}$,
    \item $\mathbf{0} \in V(T)$ is the root of $T$, and
    \item the parent of $x_0  \dots x_i 0 \dots 0$ with $x_i \neq 0$
        is $x_0 \dots (x_i - 1)0 \dots 0$.
\end{itemize}

Set $A = \{a_0, a_1, \dots, a_t\}$ with $a_0=0 \in V(T)$, and $a_i \not\in V(T)$
for every $i \in [t]$ (that is, $a_1, \dots, a_t$ are new vertices not in $T$). 
Finally, for every $i \in [t]$, for every $x=(x_0, \dots, x_{i-1}) \in [0,\ell-1]^i$,
link $a_i$ to $x_0 \dots x_{i-2} (\ell-1)0\dots 0$ via a path of length $r/2$. 
This gives the desired graph $G$, see Figure~\ref{fig:construction_bounded_treewidth}.

Observe that $G - \{a_2, \dots, a_t\}$ is a tree, so
$G$ has treewidth at most $1+(t-1)=t$ as wanted.

Let $x=x_0 \dots x_t \in V(T)$, $d_0 = \dist_G(x,a_0)$ and
$d_i = \dist_G(x,a_i) - r/2$ for every $i \in \{1, \dots, t\}$.
One can easily check the following relations:
\begin{equation}\label{eq:system_construction}
\left\{
\begin{array}{r c l l}
d_0 &=& \sum_{i=0}^t x_i \leq r/2 & \\
d_i &=&  (\ell-1-x_{i-1}) + \sum_{j=i}^t x_j \leq r/2, &\forall i \geq 1 \\
\end{array}
\right.
\end{equation}
In particular, $\dist_G(x,a_i) \leq r$ for every $x \in V(T)$ and
$i\in\{0,\dots,t\}$.

We claim that this system is injective, that is, it has at most one solution
$x_0, \dots, x_{t}$ when $d_0, \dots, d_t$ is fixed.
Indeed, we can recover $x_0, \dots,x_t$ inductively using the following
relations:
\[
\begin{array}{r c l l}
    x_{i+1} &=& \frac{1}{2}(\ell -1 - d_{i+2} +d_0-x_0-\dots-x_i) & \quad \textrm{ for } i \in [-1,t-1]  \\
\end{array}
\]
As a consequence, $|\Pi_{r,G}[V(G) \to A]| \geq |V(T)| = \ell^{t+1}
= \tfrac{1}{(2(t+1))^{t+1}} \cdot r^{t+1}$, and the result follows as $|A| = t+1$.
\end{proof}

Observe that this construction is tight for $A$ of constant size.
Indeed, Corollary~\ref{corollary:bound_polynomial_in_A_profiles} implies
that if $G$ has treewidth at most $t$, then
$|\Pi_{r,G}[V(G) \to A]| = \bigO(r^{t+1}|A|^{t+1})$
(because $G$ is then $K_{t+2}$-minor-free).

In terms of neighborhood complexity, we deduce the following corollary from the above theorem using
Lemma~\ref{lemma:from_balls_to_profiles}.

\begin{corollary}
Let $t$ be a positive integer. 
Then, for every nonnegative integer $r$ such that $2^t(t+1)$ divides $r$, there exists a graph $G$ of treewidth at most
$t$ and $A \subseteq V(G)$ such that
\[
|\{A \cap N^r[v] \mid v \in V(G)\}| 
\geq \tfrac{1}{2(2(t+1))^{t+1}} \cdot r^{t}|A|
= \bigOmega_t(r^t |A|).
\]
\end{corollary}

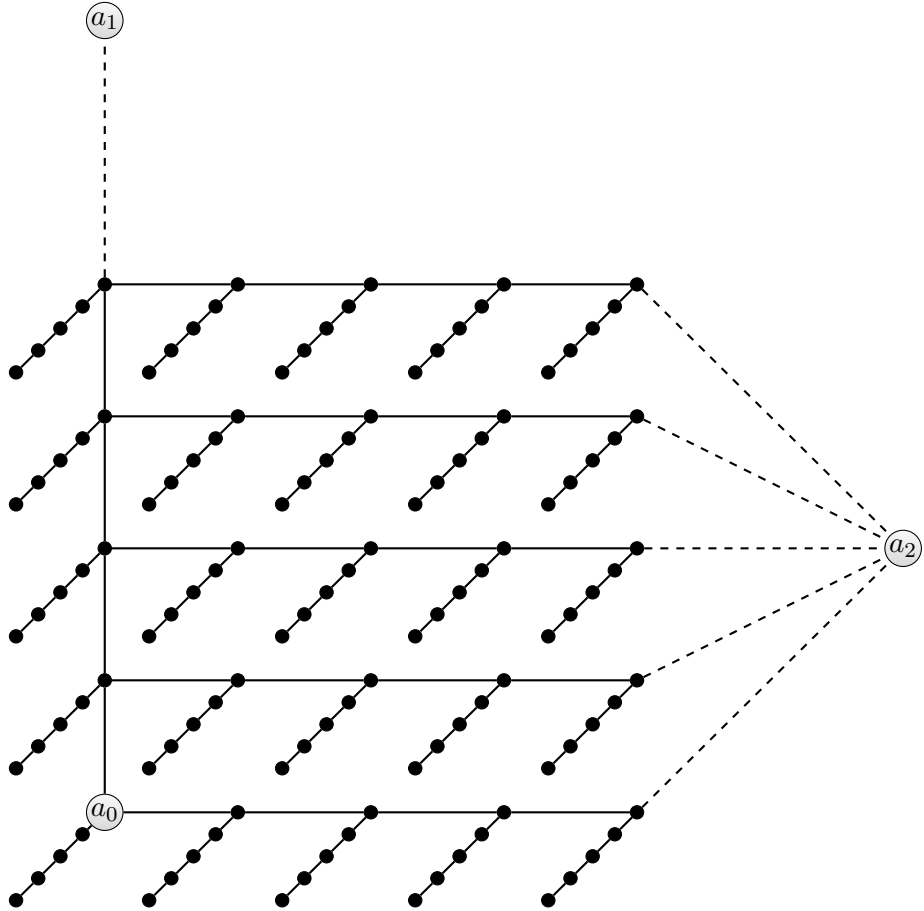
\begin{figure}[ht]
    \centering
    \begin{tikzpicture}[scale=1.75]
        \renewcommand\r{4} 

        \node (v000) at (0,0) [vertex] {$a_0$};
        \foreach \z[evaluate={\zm=int(\z-1);}] in {1,...,\r} 
        {
            \node (v00\z) at (-\z/\r/3*2, -\z/\r/3*2) [vertexverysmall]{};
            \draw[edge] (v00\zm) -- (v00\z);
        }
        \foreach \y[evaluate={\ym=int(\y-1);}] in {1,...,\r}
        {
            \node (v0\y0) at (0,\y) [vertexverysmall] {};
            \draw[edge] (v0\ym0) -- (v0\y0);
            \foreach \z[evaluate={\zm=int(\z-1);}] in {1,...,\r} 
            {
                \node (v0\y\z) at (-\z/\r/3*2, \y-\z/\r/3*2) [vertexverysmall]{};
                \draw[edge] (v0\y\zm) -- (v0\y\z);
            }
        }
        \foreach \x[evaluate={\xm=int(\x-1);}] in {1,...,\r}
        {
            \foreach \y in {0,...,\r}
            {
                \node (v\x\y0) at (\x,\y) [vertexverysmall] {};
                \draw[edge] (v\xm\y0) -- (v\x\y0);
                \foreach \z[evaluate={\zm=int(\z-1);}] in {1,...,\r} 
                {
                    \node (v\x\y\z) at (\x-\z/\r/3*2, \y-\z/\r/3*2) [vertexverysmall]{};
                    \draw[edge] (v\x\y\zm) -- (v\x\y\z);
                }
            }
        }


        \node (a1) at (0,1.5*\r) [vertex] {$a_1$};
        \draw[edge, dashed] (v0\r0) -- (a1);

        \node (a2) at (1.5*\r, \r/2) [vertex] {$a_2$};
        \foreach \y in {0,...,\r}
        {
            \draw[edge, dashed] (v\r\y0) -- (a2);
        }
    \end{tikzpicture}
    \caption{The construction of Theorem~\ref{theorem:construction_bounded_treewidth}
        for $t=2$ and $\ell=5$. The dashed edges represent paths of length $r/2$, and the central tree is a subgraph of an $\ell \times \ell \times \ell$-grid where $\ell = \frac{r}{2(t+1)}$.}
    \label{fig:construction_bounded_treewidth}
\end{figure}

\subsection{\texorpdfstring{$1$}{1}-planar graphs}
Given a nonnegative integer $k$, a graph $G$ is said to be \emph{$k$-planar} if it can be drawn in the plane so that
every edge of $G$ crosses at most $k$ other edges. Note that $0$-planar graphs are
exactly planar graphs.
However, for $k\geq 1$, the class of $k$-planar excludes no graph as a minor.

In this subsection, we show that there is no polynomial function $f$ such that the neighborhood complexity of every $1$-planar graph on a set $A$ of vertices at distance $r$ is at most $f(r)|A|$.
This also shows that the existence of a product structure is not enough to guarantee polynomial neighborhood complexity, despite Theorem~\ref{theorem:guarding_product_general_form} 
which reduces to the case where $|A|$ is bounded by a polynomial in $r$.
Indeed, Dujmović, Morin and Wood~\cite{dujmovic_graph} proved that every $k$-planar graph is a subgraph of $H \boxtimes P \boxtimes K_{18k^2+48k+30}$ for some path $P$, and some graph $H$ of treewidth at most $\binom{k+4}{3}-1$.

\begin{theorem}\label{theorem:construction_1_planar}
For every positive integer $r$ such that $r+1$ is a perfect square, there exists a $1$-planar
graph $G$ and $A \subseteq V(G)$ such that $|\{N^r[v] \cap A \mid v \in V(G)\}|
\geq \frac{2^{\sqrt{r+1}}}{\sqrt{r+1}}|A|$.
\end{theorem}

\begin{proof}
Let $\ell = \sqrt{r+1}$, and let $A = \{a_1, \dots, a_\ell\}$ be a set of $\ell$ vertices. 
For every $i \in [\ell]$, we consider a complete binary tree $T_i$ of depth $\ell$, and
with leaves $\{a_i^X\}_{X \subseteq [\ell]}$ (this indexing is possible because
a complete binary tree of depth $h$ has $2^h$ leaves).
For every $X \subseteq [\ell]$, we add a vertex $v_X$ with neighborhood
$N(v_X) = \{a_i^X \mid i \in X\}$.
This gives an $(\ell-1)$-planar graph $G_\ell$, see Figure~\ref{fig:construction_1_planar}.
Moreover, $\dist(a_i^X,a_i)=\ell$ for every $i \in [\ell]$ and $X \subseteq [\ell]$,
hence $\dist(v_X,a_i) \leq \ell+1$ if and only if $i \in X$.
It follows that for every $X\subseteq [\ell]$, $N^{\ell+1}[v_X] \cap A = \{a_i \mid i \in X\}$,
and so $|\{N^{\ell+1}[v] \cap A \mid v \in V(G_\ell)\}| \geq 2^\ell$.

To convert this graph $G_\ell$ into a $1$-planar graph, we subdivide every edge
$\ell-2$ times, and the new distance is now $(\ell+1)(\ell-1)=r$. 
Therefore, we obtain a $1$-planar graph $G$ such that 
$|\{N^r[v] \cap A \mid v \in V(G)\}| \geq 2^\ell =
\frac{2^{\sqrt{r+1}}}{\sqrt{r+1}}|A|$.
\end{proof}

Observe that in this construction, $A$ is shattered by the balls of radius $r$.
Hence this shows that there are $1$-planar graphs whose hypergraph of $r$-balls
has VC-dimension at least $\sqrt{r+1}$, whereas for $K_t$-minor-free graphs
this hypergraph has VC-dimension at most $t-1$ (which is independent of $r$)
by Theorem~\ref{theorem:VC_dim_Kt_minor_free}.




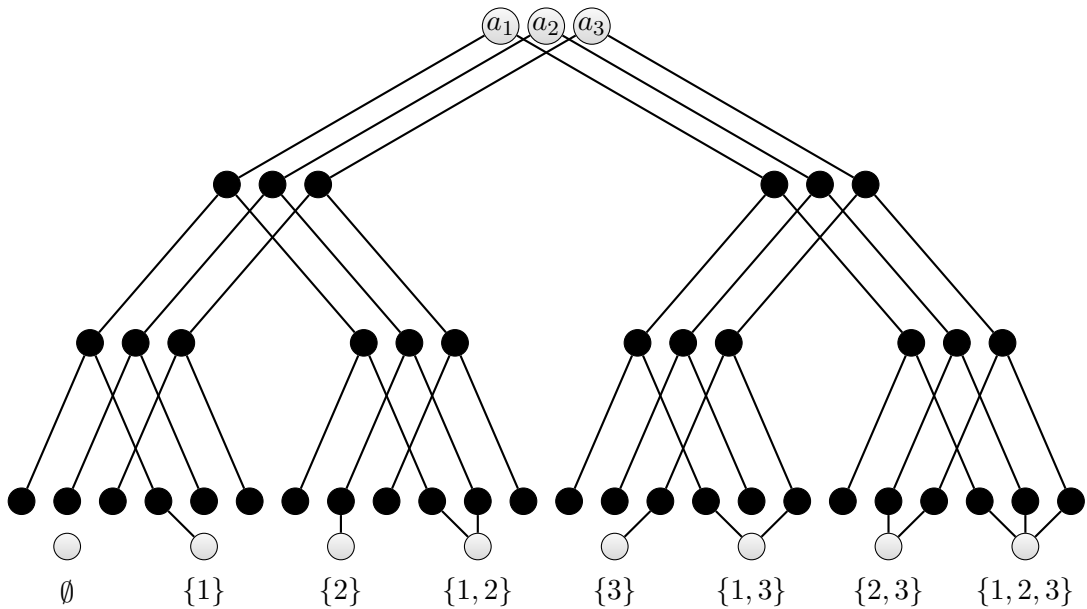
\begin{figure}[ht]
    \centering
    \begin{tikzpicture}[scale=0.6]
        \newcommand\deltay{3.5}
        \node (a1) at (-1,0) [vertex] {$a_1$};
        \node (a2) at (0,0) [vertex] {$a_2$};
        \node (a3) at (+1,0) [vertex] {$a_3$};
        
        \node (a10) at ($(a1)+(-6,-\deltay)$) [vertexsmall] {};
        \node (a20) at ($(a2)+(-6,-\deltay)$) [vertexsmall] {};
        \node (a30) at ($(a3)+(-6,-\deltay)$) [vertexsmall] {};
        \draw[edge] (a1) -- (a10);
        \draw[edge] (a2) -- (a20);
        \draw[edge] (a3) -- (a30);
            \node (a100) at ($(a10)+(-3,-\deltay)$) [vertexsmall] {};
            \node (a200) at ($(a20)+(-3,-\deltay)$) [vertexsmall] {};
            \node (a300) at ($(a30)+(-3,-\deltay)$) [vertexsmall] {};
            \draw[edge] (a10) -- (a100);
            \draw[edge] (a20) -- (a200);
            \draw[edge] (a30) -- (a300);
                \node (a1000) at ($(a100)+(-1.5,-\deltay)$) [vertexsmall] {};
                \node (a2000) at ($(a200)+(-1.5,-\deltay)$) [vertexsmall] {};
                \node (a3000) at ($(a300)+(-1.5,-\deltay)$) [vertexsmall] {};
                \draw[edge] (a100) -- (a1000);
                \draw[edge] (a200) -- (a2000);
                \draw[edge] (a300) -- (a3000);
                
                \node (a1001) at ($(a100)+(1.5,-\deltay)$) [vertexsmall] {};
                \node (a2001) at ($(a200)+(1.5,-\deltay)$) [vertexsmall] {};
                \node (a3001) at ($(a300)+(1.5,-\deltay)$) [vertexsmall] {};
                \draw[edge] (a100) -- (a1001);
                \draw[edge] (a200) -- (a2001);
                \draw[edge] (a300) -- (a3001);
            \node (a101) at ($(a10)+(3,-\deltay)$) [vertexsmall] {};
            \node (a201) at ($(a20)+(3,-\deltay)$) [vertexsmall] {};
            \node (a301) at ($(a30)+(3,-\deltay)$) [vertexsmall] {};
            \draw[edge] (a10) -- (a101);
            \draw[edge] (a20) -- (a201);
            \draw[edge] (a30) -- (a301);
                \node (a1010) at ($(a101)+(-1.5,-\deltay)$) [vertexsmall] {};
                \node (a2010) at ($(a201)+(-1.5,-\deltay)$) [vertexsmall] {};
                \node (a3010) at ($(a301)+(-1.5,-\deltay)$) [vertexsmall] {};
                \draw[edge] (a101) -- (a1010);
                \draw[edge] (a201) -- (a2010);
                \draw[edge] (a301) -- (a3010);
                
                \node (a1011) at ($(a101)+(1.5,-\deltay)$) [vertexsmall] {};
                \node (a2011) at ($(a201)+(1.5,-\deltay)$) [vertexsmall] {};
                \node (a3011) at ($(a301)+(1.5,-\deltay)$) [vertexsmall] {};
                \draw[edge] (a101) -- (a1011);
                \draw[edge] (a201) -- (a2011);
                \draw[edge] (a301) -- (a3011);
        
        \node (a11) at ($(a1)+(6,-\deltay)$) [vertexsmall] {};
        \node (a21) at ($(a2)+(6,-\deltay)$) [vertexsmall] {};
        \node (a31) at ($(a3)+(6,-\deltay)$) [vertexsmall] {};
        \draw[edge] (a1) -- (a11);
        \draw[edge] (a2) -- (a21);
        \draw[edge] (a3) -- (a31);
            \node (a110) at ($(a11)+(-3,-\deltay)$) [vertexsmall] {};
            \node (a210) at ($(a21)+(-3,-\deltay)$) [vertexsmall] {};
            \node (a310) at ($(a31)+(-3,-\deltay)$) [vertexsmall] {};
            \draw[edge] (a11) -- (a110);
            \draw[edge] (a21) -- (a210);
            \draw[edge] (a31) -- (a310);
                \node (a1100) at ($(a110)+(-1.5,-\deltay)$) [vertexsmall] {};
                \node (a2100) at ($(a210)+(-1.5,-\deltay)$) [vertexsmall] {};
                \node (a3100) at ($(a310)+(-1.5,-\deltay)$) [vertexsmall] {};
                \draw[edge] (a110) -- (a1100);
                \draw[edge] (a210) -- (a2100);
                \draw[edge] (a310) -- (a3100);
                
                \node (a1101) at ($(a110)+(1.5,-\deltay)$) [vertexsmall] {};
                \node (a2101) at ($(a210)+(1.5,-\deltay)$) [vertexsmall] {};
                \node (a3101) at ($(a310)+(1.5,-\deltay)$) [vertexsmall] {};
                \draw[edge] (a110) -- (a1101);
                \draw[edge] (a210) -- (a2101);
                \draw[edge] (a310) -- (a3101);
            \node (a111) at ($(a11)+(3,-\deltay)$) [vertexsmall] {};
            \node (a211) at ($(a21)+(3,-\deltay)$) [vertexsmall] {};
            \node (a311) at ($(a31)+(3,-\deltay)$) [vertexsmall] {};
            \draw[edge] (a11) -- (a111);
            \draw[edge] (a21) -- (a211);
            \draw[edge] (a31) -- (a311);
                \node (a1110) at ($(a111)+(-1.5,-\deltay)$) [vertexsmall] {};
                \node (a2110) at ($(a211)+(-1.5,-\deltay)$) [vertexsmall] {};
                \node (a3110) at ($(a311)+(-1.5,-\deltay)$) [vertexsmall] {};
                \draw[edge] (a111) -- (a1110);
                \draw[edge] (a211) -- (a2110);
                \draw[edge] (a311) -- (a3110);
                
                \node (a1111) at ($(a111)+(1.5,-\deltay)$) [vertexsmall] {};
                \node (a2111) at ($(a211)+(1.5,-\deltay)$) [vertexsmall] {};
                \node (a3111) at ($(a311)+(1.5,-\deltay)$) [vertexsmall] {};
                \draw[edge] (a111) -- (a1111);
                \draw[edge] (a211) -- (a2111);
                \draw[edge] (a311) -- (a3111);
    
    \node (v000) at ($(a2000)+(0,-1)$) [vertex] {};
    \node at ($(v000) + (0,-1)$) {$\emptyset$};

    \node (v001) at ($(a2001)+(0,-1)$) [vertex] {};
    \draw[edge] (v001) -- (a1001);
    \node at ($(v001) + (0,-1)$) {$\{1\}$};

    \node (v010) at ($(a2010)+(0,-1)$) [vertex] {};
    \draw[edge] (v010) -- (a2010);
    \node at ($(v010) + (0,-1)$) {$\{2\}$};

    \node (v011) at ($(a2011)+(0,-1)$) [vertex] {};
    \draw[edge] (v011) -- (a1011);
    \draw[edge] (v011) -- (a2011);
    \node at ($(v011)+(0,-1)$) {$\{1,2\}$};

    \node (v100) at ($(a2100)+(0,-1)$) [vertex] {};
    \draw[edge] (v100) -- (a3100);
    \node at ($(v100) + (0,-1)$) {$\{3\}$};

    \node (v101) at ($(a2101)+(0,-1)$) [vertex] {};
    \draw[edge] (v101) -- (a1101);
    \draw[edge] (v101) -- (a3101);
    \node at ($(v101)+(0,-1)$) {$\{1,3\}$};

    \node (v110) at ($(a2110)+(0,-1)$) [vertex] {};
    \draw[edge] (v110) -- (a2110);
    \draw[edge] (v110) -- (a3110);
    \node at ($(v110)+(0,-1)$) {$\{2,3\}$};

    \node (v111) at ($(a2111)+(0,-1)$) [vertex] {};
    \draw[edge] (v111) -- (a1111);
    \draw[edge] (v111) -- (a2111);
    \draw[edge] (v111) -- (a3111);
    \node at ($(v111)+(0,-1)$) {$\{1,2,3\}$};
    \end{tikzpicture}
    \caption{The graph $G_\ell$ of Theorem~\ref{theorem:construction_1_planar}
    for $\ell=3$. This graph is $(\ell-1)$-planar and has neighborhood complexity on
    $A=\{a_1,a_2,\dots, a_\ell\}$ at distance $\ell+1$ at least $2^\ell/\ell$. }
    \label{fig:construction_1_planar}
\end{figure}

\section{Application to metric dimension}
\label{sec:metric_dimension}

In a connected graph $G$, a \emph{resolving set} $S$ is a set of vertices such that no two vertices have the same profile on $S$ at distance $\diam(G)$, where $\diam(G)$ denotes the diameter of $G$.
The \emph{metric dimension} of $G$ is the smallest size of a resolving set. 
Clearly, if $G$ has metric dimension $k$ and diameter $d$, then the number $n$ of vertices of $G$ is bounded as follows:
\[
n\leq d^k+k.
\]
In~\cite{beaudou_bounding_2018}, the authors improved this upper bound for several particular graph classes.
\begin{theorem}[Beaudou, Foucaud, Dankelmann, Henning, Mary, and Parreau~\cite{beaudou_bounding_2018}]
Let $G$ be an $n$-vertex graph with diameter $d$ and metric dimension $k$.
\begin{enumerate}
    \item If $G$ is a tree, then $n \in \bigO(d^2k)$,
    \item if $G$ is $K_t$-minor free, then $n \in \bigO((dk)^{t-1})$,
    \item if $G$ has treewidth $t$, then $n \in \bigO(d^{3t+3} k)$.
\end{enumerate}
\end{theorem}

By observing that $n \leq |\Pi_{d,G}[V(G) \to S]|$ if $S$ is a resolving set, we obtain the following bounds.
\begin{theorem}
\label{thm:metric_dimension}
Let $G$ be a connected $n$-vertex graph with diameter $d$ and metric dimension $k$.
\begin{enumerate}
    \item If $G$ is in a fixed class $\mathcal{C}$ of bounded expansion, then $n \leq f(d)k$ for some function $f$ depending only on $\mathcal{C}$ (by~\cite{RSS19} and Lemma~\ref{lemma:from_balls_to_profiles}).
    \item If $G$ is planar, then $n \in \bigO(d^4k)$ (by Theorem~\ref{theorem:NC_planar_degree_4}).
    \item If $G$ has Euler genus $g$, then $n \in \bigO_g(d^4(d+k))$ (by Theorem~\ref{theorem:bounded_genus}).
    \item If $G$ is $K_t$-minor-free, then $n \in \bigO_t(d^{t^2-1} k)$ (by Theorem~\ref{thm:NC_Kt_Minor_free}).
    \item If $G$ has treewidth $t \geq 1$, then $n \in \bigO(d^{2t} k)$ (by Corollary~\ref{corollary:NC_bounded_treewidth}).
\end{enumerate}
\end{theorem}

\section{Conclusion}
\label{sec:conclusion}

We conclude the paper with two open problems. 
The first one is about getting a tight bound on the profile complexity of planar graphs, since the existing lower and upper bounds are tantalizing close to each other. 

\begin{problem}
Determine the profile and neighborhood complexities of planar graphs up to a constant factor. 
For profile complexity, it is known to be in $\Omega(r^3)$ (Theorem~\ref{theorem:construction_bounded_treewidth}) and in $\bigO(r^4)$ (Theorem~\ref{thm:profile_complexity_planar}).  
For neighborhood complexity, it is known to be in $\Omega(r^2)$ (Theorem~\ref{theorem:construction_bounded_treewidth} and Lemma~\ref{lemma:from_balls_to_profiles}) and in $\bigO(r^4)$ (Theorem~\ref{thm:profile_complexity_planar}).  
\end{problem}

The second open problem is about getting better bounds for $K_t$-minor-free graphs.  

\begin{problem}
Are the profile and neighborhood complexities of $K_t$-minor-free graphs in $r^{\bigO(t)}$?
\end{problem}

\section*{Acknowledgments}
We are grateful to Jacob Holme, Erik Jan van Leeuwen, and Marcin Pilipczuk for their feedback on an earlier version of the paper, which lead to the improved bound of $\bigO(r^4)$ on the neighborhood complexity of planar graphs presented in Section~\ref{section:planar}.  
We also thank Stéphan Thomassé for helpful discussions, and the two anonymous referees for their very helpful remarks on an earlier version of the paper.

\bibliographystyle{abbrv}
\bibliography{biblio}

\end{document}